\documentclass{amsart}
\usepackage{graphics}
\usepackage{amscd}
\usepackage{mathdots}
\usepackage{amsfonts,amssymb,latexsym,color}
\usepackage{xypic}
\newtheorem{theorem}{Theorem}
\newtheorem{proposition}[theorem]{Proposition}
\newtheorem{corollary}[theorem]{Corollary}
\newtheorem{lemma}[theorem]{Lemma}

\theoremstyle{definition}
\newtheorem{definition}[theorem]{Definition}
\newtheorem{remark}[theorem]{Remark}

\newtheorem*{notation}{Notation}
\newcommand{\pp}{\mathbb{P}}
\newcommand{\nn}{\mathbb{N}}
\newcommand{\qq}{\mathbb{Q}}
\newcommand{\zz}{\mathbb{Z}}

\newcommand{\amplebu}[1]{\mathcal{O}_{X}({#1})}

\newcommand{\linebu}{\mathsf{N}}

\newcommand{\fr}{$\eps$}
\newcommand{\eps}{\mathbf{\varepsilon}}
\newcommand{\epss}[1]{\eps_{\scriptscriptstyle{#1}}}
\newcommand{\deltabar}{\overline{\delta}}

\newcommand{\Hom}{\text{Hom}}

\newcommand{\pic}{\mbox{Pic}}

\newcommand{\rk}[1]{\text{rk}({#1})}
\newcommand{\rkk}[1]{r_{\scriptscriptstyle{#1}}}
\newcommand{\coker}{\text{co}\ker}
\newcommand{\rest}[2]{{#1}_{\mid_{#2}}}
\newcommand{\tipo}{{\underline{\mathfrak{t}}}}
\newcommand{\ab}{_{a,b}}
\newcommand{\abc}{_{a,b,c}}
\newcommand{\kk}[2]{\textsf{k}_{\scriptscriptstyle{#1},{#2}}}
\newcommand{\kkvoid}{\textsf{k}}

\newcommand{\dual}[1]{{#1}^{\vee}}

\newcommand{\ggeq}{\stackrel{\scriptscriptstyle >}{\null_{\scriptscriptstyle{(=)}}}}
\newcommand{\lleq}{\stackrel{\scriptscriptstyle <}{\null_{\scriptscriptstyle{(=)}}}}
\newcommand{\ppreceq}{\stackrel{\scriptscriptstyle \prec}{\null_{\scriptscriptstyle (-)}}}
\newcommand{\ssucceq}{\stackrel{\scriptscriptstyle \succ}{\null_{\scriptscriptstyle (-)}}}
\newcommand{\efilcal}{\ecal^{\bullet}}
\newcommand{\alphafil}{\underline{\alpha}}

\newcommand{\filtrationcal}{\efilcal,\alphafil}
\newcommand{\ttI}{\mathtt{I}}
\newcommand{\ttIbar}{\overline{\mathtt{I}}}
\newcommand{\ttJ}{\mathtt{J}}
\newcommand{\pttI}{P_\ttI}
\newcommand{\muttI}{\mu_\ttI}
\newcommand{\rttI}{R_\ttI}
\newcommand{\gammattI}{\gamma_{\ttI}}

\newcommand{\hilpol}[1]{\mathsf{P}_{\scriptscriptstyle{#1}}}
\newcommand{\hilpolred}[1]{\mathsf{p}_{\scriptscriptstyle{#1}}}
\newcommand{\ecal}{\mathcal{E}}
\newcommand{\fcal}{\mathcal{F}}
\newcommand{\gcal}{\mathcal{G}}

\newcommand{\wcal}{\mathcal{W}}

\newcommand{\acal}{\mathcal{A}}
\newcommand{\lcal}{\mathcal{L}}
\newcommand{\qcal}{\mathcal{Q}}
\newcommand{\ocal}{\mathcal{O}}

\newcommand{\tcal}{\mathcal{T}}
\newcommand{\quot}{\mathsf{Q}{\scriptscriptstyle{{\text{uot}}}}}
\newcommand{\quotfunctor}{\underline{\mathsf{Q}{\scriptscriptstyle{{\text{uot}}}}}}

\newcommand{\qttfunctor}{\underline{\mathsf{Q}}}

\newcommand{\ai}{\mathsf{a}}
\newcommand{\vai}{{V_{\mathsf{a}}}}

\newcommand{\dmu}{\mu_{\eps}}
\newcommand{\ddeg}{\deg_{\eps}}
\newcommand{\kker}{\mathsf{K}}

\newcommand{\wtild}[1]{\widetilde{#1}}

\newcommand{\deter}{\gcal_\ai} 

\newcommand{\sudi}[1]{{#1}_D}
\newcommand{\degq}{\deg^q}
\newcommand{\muq}{\mu^q}

\newcommand{\ddegq}{\ddeg^q}
\newcommand{\dmuq}{\dmu^q}
\newcommand{\rkq}[1]{{\text{rk}^q({#1})}}
\newcommand{\rqq}{{r^q}}
\author{Andrea Pustetto}
\email{pustetto@gmail.com}
\title{Metha-Ramanathan for $\eps$ and $\kkvoid$-semistable Decorated Sheaves}
\date{\today}
\begin{document}

\maketitle

\begin{abstract}
This paper is devoted to generalizing the Mehta-Ramanathan restriction theorem to the case of $\eps$-semistable and $\kkvoid$-semistable decorated sheaves. After recalling the definition of decorated sheaves and their usual semistability we define the $\eps$ and $\kkvoid$-(semi)stablility. We first prove the existence of a (unique) \fr-maximal destabilizing subsheaf for decorated sheaves (Section \ref{maximal-dest-subsheaf}). After some others preliminar results (such as the opennes condition for families of \fr-semistable decorated sheaves) we finally prove, in Section \ref{sec-restriction-theorem}, a restriction theorem for slope \fr-semistable decorated sheaves. In Section \ref{sec-mehta-ramanathan-2} we reach the same results in the $\kkvoid$-semistability case that we did in the \fr-semistability, but only for rank $\leq3$ decorated sheaves.
\end{abstract}

\section{Introduction}
In the framework of bundles with a decoration we recall two types of objects which incorporate all others: \textbf{decorated sheaves} and the so-called \textbf{tensors}. The former were introduced by Schmitt while the latter by Gomez and Sols. We recall briefly the definitions of such objects. A \textit{decorated sheaf} of type $(a,b,c,\linebu)$ over $X$ is a pair $(\ecal,\varphi)$ where $\ecal$ is a torsion free sheaf over $X$ and
\begin{equation*}
 \varphi:(\ecal^{\otimes a})^{\oplus b}\otimes (\det\dual{\ecal})^{\otimes c}\longrightarrow\linebu,
\end{equation*}
while a \textit{tensor} of type $(a,b,c,D_u)$ is a pair $(\ecal,\varphi)$ where $\ecal$ is a torsion free sheaf and
\begin{equation*}
 \varphi:(\ecal^{\otimes a})^{\oplus b}\otimes\longrightarrow(\det\ecal)^{\otimes c}\otimes D_u,
\end{equation*}
where $D_u$ is a locally free sheaf belonging to a fixed family $\{D_u\}_{u\in R}$ parametrized by a scheme $R$. As one can easily see, these two objects are quite similar; they both incorporate many types of bundles with a morphism, such as framed bundles, Higgs bundles, quadratic, orthogonal and symplectic bundles, and many others. The problem with classifying decorated sheaves up to equivalence is therefore related to many classification problems in algebraic geometry. In order to solve this problem by establishing the existence of a coarse moduli space one needs to introduce a notion of semistability. The notion of semistability for decorated sheaves and tensors is the same. A semistability condition for these objects was introduced by Bradlow, Garcia-Prada, Gothen and Mundet i Riera in \cite{BGPM} and \cite{GPGMR2} and then studied by Schmitt in the case of decorated bundle (see for example \cite{Sch}) and by Gomez-Sols in the case of tensors (see for exaple \cite{GS-1}).\\

In both cases one tests the (semi)stability of an object $(\ecal,\varphi)$ against \textbf{saturated weighted fitrations} of $\ecal$, namely against pairs $(\efilcal,\alphafil)$ consisting of a filtration
\begin{equation*}
 \efilcal\,:\quad 0\subset\ecal_1\subset\dots\subset\ecal_s\subset\ecal_{s+1}=\ecal
\end{equation*}
of saturated sheaves of $\ecal$, i.e., such that $\ecal/\ecal_i$ is torsion free, and a tuple
\begin{equation*}
 \alphafil = (\alpha_1,\dots,\alpha_s)
\end{equation*}
of positive rational numbers. Then one says that a decorated sheaf or tensor $(\ecal,\varphi)$ is (semi)stable with respect to $\delta$ if and only if for any weighted filtration 
\begin{equation*}
P(\efilcal,\alphafil)+\delta\;\mu(\efilcal,\alphafil;\varphi)\ssucceq0,
\end{equation*}
where $P(\efilcal,\alphafil)$ is a polynomial depending on the Hilbert polynomials of the sheaves of the weighted filtration, $\delta$ is a fixed polynomial and $\mu(\efilcal,\alphafil;\varphi)$ a number depending on the weighted filtration and on $\varphi$ (see Definition \ref{def-ss}).\\

As one can easily observe when looking at the definition of $\mu(\efilcal,\alphafil;\varphi)$, the semistability for decorated sheaves is not easy to handle and is quite complicated to calculate in general. This fact affects the possibility of generalizing many basic tools that instead exist for vector bundles. For example, until now, there is no notion of a maximal destabilizing object for decorated sheaves nor is there a notion of Jordan-H\"older or Harder-Narasimhan filtration although some results in this direction can be found in \cite{GS-1} for similar objects. This paper is devoted to the study of the semistability condition of decorated bundles in order to better understand and simplify it in the hope that this will be useful in the study of decorated sheaves. In \cite{AlePu2} we tried, and succeded for the case of $a=2$, to bound the length of weighted filtrations on which one checks the semistability condition, while in this paper we approach the problem in a different way: we ``enclose'' the above semistability condition between a stronger semistability condition and a weaker one. To be more precise: we say that a decorated sheaf $(\ecal,\varphi)$ of type $(a,b,c,\linebu)$ is \textbf{$\eps$-(semi)stable} with respect to a fixed polynomial $\delta$ of degree $\dim X -1$ if for any subsheaf $\fcal\subset\ecal$
\begin{equation*}
 \rk{\ecal}\left(\hilpol{\fcal}-a\,\delta\,\eps(\rest{\varphi}{\fcal})\right)\ppreceq\rk{\fcal}\left(\hilpol{\ecal}-a\,\delta\,\eps(\varphi)\right),
\end{equation*}
where $\hilpol{\fcal}$ denotes the Hilbert polynomial of a sheaf $\fcal$ and $\eps(\varphi)=1$ if $\varphi\neq0$ and zero otherwise. Similarly, given a sheaf $\ecal$ and a subsheaf $\fcal\subset\ecal$, we define a function $\kk{\fcal}{\ecal}$ with values in the set $\{0,1,\dots,a\}$ and depending on the behaviour of $\varphi$ on $\fcal$ as subsheaf of $\ecal$ (see Equation \eqref{eq-def-kk}). Then we say that a decorated sheaf $(\ecal,\varphi)$ is \textbf{$\kkvoid$-(semi)stable} if for any subsheaf $\fcal\subset\ecal$ one has that
\begin{equation*}
  \rk{\ecal}\left(\hilpol{\fcal}-\delta\,\kk{\fcal}{\ecal})\right)\ppreceq\rk{\fcal}\left(\hilpol{\ecal}-\delta\,\kk{\ecal}{\ecal}\right).
\end{equation*}
What happens is that
\begin{equation*}
 \eps\text{-(semi)stable} \;\Rightarrow\; \text{(semi)stable} \;\Rightarrow\; \kkvoid\text{-(semi)stable}
\end{equation*}
and, if $\rk{\ecal}=2$, (semi)stability is equivalent to $\kkvoid$-(semi)stability. In this respect, we generalize some known results (in the case of vector bundles) to the case of \fr-semistable decorated sheaf and, to a lesser extent, to the $\kkvoid$-(semi)stable case. In fact, using \fr-semistability, we find the \fr-maximal destabilizing subsheaf, prove a Mehta-Ramanathan's like theorem about the behavior of slope \fr-semistability under restriction to curves. Since $\kkvoid$-semistability is a little bit more complicated to handle, we managed to find a $\kkvoid$-maximal destabilizing subsheaf and prove a Mehta-Ramanathan theorem only for rank $\leq3$.\\
\section{Definition and first properties}\label{section-decorated-sheaves}
\begin{notation}
Let $(X,\amplebu{1})$ be a polarized projective smooth variety of dimension $n$, $\delta=\delta(x)\doteqdot\delta_{n-1}x^{n-1}+\dots+\delta_{1}x+\delta_{0}\in\qq[x]$ be a fixed polynomial with positive leading coefficient and let $\deltabar=\delta_{n-1}$. Recall that given two polynomials $p(m)$ and $q(m)$ then $p\preceq q$ if and only if there exists $m_0\in\nn$ such that $p(m)\leq q(m)$ for any $m\geq m_0$.
\end{notation}
\begin{definition}\label{def-decorato}
Let $\linebu$ be a line bundle over $X$ and let $a,b,c$ be nonnegative integers. A \textbf{decorated vector bundle} of type $\tipo=(a,b,c,\linebu)$ over $X$ is the datum of a vector bundle $E$ over $X$ and a morphism
\begin{equation}\label{eq-def-decorato}
\varphi:E\abc\doteqdot (E^{\otimes a})^{\oplus b}\otimes(\det E)^{\otimes -c}\longrightarrow\linebu
\end{equation}
A \textbf{decorated sheaf} of type $\tipo$ is instead a pair $(\ecal,\varphi)$ such that $\ecal$ is a torsion free sheaf and $\varphi$ is a morphism as in (\ref{eq-def-decorato}).
Sometimes we will call these objects simply decorated sheaves (respectively bundles) instead of decorated sheaves (resp. bundles) of type $\tipo=(a,b,c,\linebu)$ if the input data are understood.
\end{definition}

\begin{remark}
 Note that, although $\ecal$ is torsion free, the sheaf $\ecal\abc$ may have torsion.
\end{remark}

Now we define morphisms between such objects. Let $(\ecal,\varphi)$ and $(\ecal',\varphi')$ be decorated sheaves (resp. bundles) of the same type $\tipo$. A morphism of sheaves (resp. bundles) $f:\ecal\to\ecal'$ is a \textbf{morphism of decorated sheaves} (resp. \textbf{bundles}) if exists a scalar morphism $\lambda:\linebu\to\linebu$ making the following diagram commute:
\begin{equation}\label{eq-def-morphism-dec}
 \xymatrix{ 
\ecal\abc \ar[d]^{\varphi} \ar[r]^{f\abc} & \ecal'\abc \ar[d]^{\varphi'}\\
\linebu \ar[r]^{\lambda} & \linebu.
}
\end{equation}

We will say that a morphism of decorated sheaves (bundles) $f:(\ecal,\varphi)\to(\ecal',\varphi')$ is \textbf{injective} if exists an injective morphism of sheaves (bundles) $f:\ecal\hookrightarrow\ecal'$ and a \underline{non-zero} scalar morphism $\lambda$ such that the above diagram commutes. Analogously we will say that a morphism of decorated sheaves (bundles) $f:(\ecal,\varphi)\to(\ecal',\varphi')$ is \textbf{surjective} if exists a surjective morphism of sheaves (bundles) $f:\ecal\to\ecal'$ and a scalar morphism $\lambda$ making diagram (\ref{eq-def-morphism-dec}) commute. Finally we will say that $(\ecal,\varphi)$ and $(\ecal',\varphi')$ are \textbf{equivalent} if exists an injective and surjective morphism of decorated sheaves (bundles) between them.
\subsection{Semistability conditions}\label{sec-semistability-conditions}
Let $(\ecal,\varphi)$ be a decorated sheaf of type $\tipo=(a,b,c,\linebu)$ and let $r=\rk{\ecal}$. We want to recall the notion of semistability for these objects. To this end let
\begin{equation}
 0\subsetneq\ecal_{i_1}\subsetneq\dots\subsetneq\ecal_{i_s}\subsetneq\ecal_{r}=\ecal
\end{equation}
be a filtration of saturated subsheaves of $\ecal$ such that $\rk{\ecal_{i_j}}=i_j$ for any $j=1,\dots,s$, let $\alphafil=(\alpha_{i_1},\dots,\alpha_{i_s})$ be a vector of positive rational numbers and finally let $\ttI=\{i_1,\dots,i_s\}$ be the set of indexes appearing in the filtration. We will refer to the pair $(\filtrationcal)_\ttI$ as \textbf{weighted filtration} of $\ecal$ indexed by $\ttI$ or simply weighted filtration if the set of indexes is clear from the context, moreover we will denote by $|\ttI|$ the cardinality of the set $\ttI$. Such a weighted filtration defines the polynomial
\begin{equation}\label{eq-def-P-di-filtrazione}
 \pttI(\filtrationcal)\doteqdot\sum_{i\in\ttI}\alpha_i\left(\hilpol{\ecal}\cdot\rk{\ecal_i}-\rk{\ecal}\cdot\hilpol{\ecal_i}\right),
\end{equation}
and the rational number
\begin{equation}\label{eq-def-slope-ss1}
 L_\ttI(\filtrationcal)\doteqdot\sum_{i\in\ttI}\alpha_i\left(\deg{\ecal}\cdot\rk{\ecal_i}-\rk{\ecal}\cdot\deg{\ecal_i}\right).
\end{equation}
Moreover we associate with $(\filtrationcal)_\ttI$ the following rational number depending also on $\varphi$:
\begin{equation}\label{eq-def-mu}
 \muttI(\filtrationcal;\varphi)\doteqdot -\min_{i_1,\dots,i_a\in\ttIbar}\{ \gammattI^{(i_1)}+\dots+\gammattI^{(i_a)} \,|\, \rest{\varphi}{(\ecal_{i_1}\otimes\dots\otimes\ecal_{i_a})^{\oplus b}}\not\equiv0\}
\end{equation}
where $\ttIbar\doteqdot\ttI\cup\{r\}$, $\hilpol{\ecal}$ (respectively $\hilpol{\ecal_i}$) is the Hilbert polynomial of $\ecal$ (resp. $\ecal_i$) and
\begin{align}
  \gammattI & =(\gammattI^{(1)},\dots,\gammattI^{(r)})\nonumber\\ 
             & \doteqdot \sum_{i\in\ttI} \alpha_i (\underbrace{\rk{\ecal_i}-r,\dots,\rk{\ecal_i}-r}_{\rk{\ecal_i}\text{-times}},\underbrace{\rk{\ecal_i},\dots,\rk{\ecal_i}}_{r-\rk{\ecal_i}\text{-times}}).
\end{align}

\begin{definition}[\textbf{Semistability}]\label{def-ss}
 A decorated sheaf $(\ecal,\varphi)$ of type $(a,b,c,\linebu)$ is $\mathbf{\delta}$-\textbf{(semi)stable} if for any weighted filtration $(\filtrationcal)$ the following inequality holds:
\begin{equation}\label{eq-def-ss}
  \pttI(\filtrationcal)+\delta\muttI(\filtrationcal;\varphi)\ssucceq0.
\end{equation}
It is \textbf{slope $\deltabar$-(semi)stable} if
\begin{equation}\label{eq-def-slope-ss}
   L_\ttI(\filtrationcal)
   +\deltabar\muttI(\filtrationcal;\varphi)\ggeq0.
\end{equation}
\end{definition}
\begin{notation}
 The notation “$\ggeq$” ($\ssucceq$) means that “$>$” (resp. $\succ$) has to be used in the definition of stable and “$\geq$” (resp. $\succeq$) in the definition of semistable.
\end{notation}
\begin{remark}\label{rem-diamond}
 \begin{enumerate}
  \item The morphism $\varphi:\ecal\abc\to\linebu$ induces a morphism $\ecal\ab\to(\det\ecal)^{\otimes c}\otimes\linebu$. With abuse of notation, we still refer to the former by $\varphi$. In this context is easy to see that a decorated sheaf of type $(a,b,c,\linebu)$ corresponds (uniquely up to isomorphism) to a decorated sheaf of type $(a,b,0,(\det\ecal)^{\otimes c}\otimes\linebu)$. Therefore the category of decorated sheaves (with fixed determinant $=\lcal$) of type $(a,b,c,\linebu)$  is equivalent to the category of decorated sheaves (with fixed determinant $=\lcal$) of type $(a,b,0,\lcal^{\otimes c}\otimes\linebu)$. For this reason if $P$ is any property witch does not involve families of decorated sheaves but just a fixed decorated sheaf then $P$ holds true for decorated sheaves of type $(a,b,c,\linebu)$ if and only if it holds true for decorated sheaves of type $(a,b,\linebu)$.
  \item Let $(\filtrationcal)$ be a weighted filtration and suppose that $\muttI=-(\gammattI^{(i_1)}+\dots+\gammattI^{(i_a)})$, then there exists at least one permutation $\sigma:\{i_1,\dots,i_a\}\to\{i_1,\dots,i_a\}$ such that $\rest{\varphi}{(\ecal_{\sigma(i_1)}\otimes\dots\otimes\ecal_{\sigma(i_a)})^{\oplus b}}\not\equiv0$. We can say that, although the morphism $\varphi$ is not symmetric, the semistability condition has a certain symmetric behavior.
  \item From now on we will write
\begin{equation*}
 \rest{\varphi}{(\ecal_{i_1}\diamond\dots\diamond\ecal_{i_a})^{\oplus b}}\not\equiv0
\end{equation*}
if there exists at least one permutation $\sigma:\{i_1,\dots,i_a\}\to\{i_1,\dots,i_a\}$ such that $\rest{\varphi}{(\ecal_{\sigma(i_1)}\otimes\dots\otimes\ecal_{\sigma(i_a)})^{\oplus b}}\not\equiv0$.
 \end{enumerate}
\end{remark}

Fix now a weighted filtration $(\filtrationcal)$ indexed by $\ttI$, define $r_s\doteqdot\rk{\ecal_s}$ and suppose that the minimum of $\muttI(\filtrationcal;\varphi)$ is attained in $i_1,\dots,i_a$. Then
\begin{align*}
 \muttI(\filtrationcal;\varphi)= & -(\gammattI^{(i_1)}+\dots+\gammattI^{(i_a)})\\
                               = & -\left( \sum_{s\in\ttI}\alpha_s r_s - \sum_{s\geq i_1}\alpha_s r + \dots + \sum_{s\in\ttI}\alpha_s r_s - \sum_{s\geq i_a}\alpha_s r\right)\\
                               = & -a\sum_{s\in\ttI}\alpha_s r_s+r\left(\sum_{s\geq i_1}\alpha_s+\dots+\sum_{s\geq i_a}\alpha_s\right).
\end{align*}
Then define
\begin{align}\label{def-R_I}
& \rttI(l)\doteqdot\sum_{s\geq l, s\in\ttI}\alpha_s \text{ for }l\in\ttI\text{ and } \rttI(r)\doteqdot0,\nonumber\\
& \rttI=\rttI(\filtrationcal;\varphi)\doteqdot\max_{i_1,\dots,i_a\in\ttIbar}\left\{ \rttI(i_1)+\dots+\rttI(i_a) \,|\, \rest{\varphi}{(\ecal_{i_1}\otimes\dots\otimes\ecal_{i_a})^{\oplus b}}\not\equiv0\right\}
\end{align}
and finally fix, for any $i\in\ttI$, the following quantities
\begin{align}\label{eq-def-costanti-ci-grandi-e-piccole}
& C_i\doteqdot r_i\hilpol{\ecal}-r\hilpol{\ecal_i}-a r_i,\\
& c_i\doteqdot r_i\deg\ecal-r\deg\ecal_i-a r_i.
\end{align}
Therefore the semistability condition (\ref{eq-def-ss}) is equivalent to the following
\begin{equation}\label{eq-def-ss-con-ci}
 \sum_{i\in\ttI} \alpha_i C_i + r\delta\rttI\ggeq0,
\end{equation}
while the slope semistability condition (\ref{eq-def-slope-ss}) is equivalent to the following
\begin{equation}\label{eq-def-slope-ss-con-ci}
 \sum_{i\in\ttI} \alpha_i c_i + r\deltabar\rttI\ggeq0.
\end{equation}
Sometimes, for convenience's sake, we will write $\rttI(i_1,\dots,i_a)$ instead of $\rttI(i_1)+\dots+\rttI(i_a)$.
\subsection{Others notions of semistabilities}\label{sec-others-notions-ss}
Decorated sheaves provide a useful a machinery to study principal bundles or more generally vector bundles with additional structures. However in general it is quite hard to check semistability for decorated sheaves because one has to verify inequality \eqref{eq-def-ss} and therefore calculate $\muttI$ for any filtration. For this reason we introduce two notions of semistability for decorated sheaves which are more computable: \fr-semistability and $\kkvoid$-semistability. \fr-semistability is quite similar to the semistability condition for framed sheaves given by Huybrechts and Lehn in \cite{HL1} while $\kkvoid$-semistability is the usual semistability restricted to subsheaves insted of filtrations. We will prove that \fr-semistability is stronger than the usual one while $\kkvoid$-semistability is weaker (Proposition \ref{prop-eps-implica-ss-implica-k}).\\

Let $(\ecal,\varphi)$ be as before and let $\fcal$ be a subsheaf of $\ecal$ then define
\begin{equation}
 \epss{\fcal}=\eps(\fcal,\varphi)\doteqdot\begin{cases}
                                   1 \text{ if } \rest{\varphi}{\fcal\ab}\not\equiv0\\
                                   0 \text{ otherwise,}
                                  \end{cases}
 \end{equation}
and
\begin{equation}\label{eq-def-kk}
  \kk{\fcal}{\ecal}=\kkvoid(\fcal,\ecal,\varphi)\doteqdot\begin{cases}
                                                          a \text{ if } \rest{\varphi}{\fcal\ab}\not\equiv0\\
                                                          a-s \text{ if } \rest{\varphi}{\fcal^{\diamond(a-s)}\diamond\ecal^{\diamond s}}\not\equiv0 \text{ and }\rest{\varphi}{\fcal^{\diamond(a-s+1)}\diamond\ecal^{\diamond s-1}}\equiv0\\
                                                          0 \text{ otherwise,} 
                                                         \end{cases}
\end{equation}
Here with the notation $\fcal^{\diamond(a-s)}\diamond\ecal^{\diamond s}$ we mean any tensor product between $\ecal$ and $\fcal$ where $\ecal$ appears exactly $s$-times while $\fcal$ appears $a-s$-times, and when we write $\rest{\varphi}{(\fcal^{\diamond a-s}\diamond\ecal^{\diamond s})^\oplus b}\not\equiv0$ we mean that there exists at least one tensor product between $\fcal$ and $\ecal$ over which $\varphi$ is not identically zero (see Remark \ref{rem-diamond} point $(3)$).\\

\begin{definition}[\textbf{$\eps$-semistability, $\kkvoid$-semistability}]
Let $(\ecal,\varphi)$ be a decorated sheaf; we will say that $(\ecal,\varphi)$ is $\mathbf{\eps}$-(semi)stable, slope $\mathbf{\eps}$-(semi)stable, $\mathbf{\kkvoid}$-(semi)stable or slope $\mathbf{\kkvoid}$-(semi)stable if and only if for any subsheaf $\fcal$ the following inequalities hols:
\begin{align}
& \mathbb{\eps}\textbf{-(semi)stable } \quad & \hilpolred{\fcal}-\frac{a\delta\epss{\fcal}}{\rk{\fcal}} & \ppreceq\hilpolred{\ecal}-\frac{a\delta}{\rk{\ecal}}\label{eq-def-eps-ss}\\
& \textbf{slope }\mathbb{\eps}\textbf{-(semi)stable } \quad & \mu(\fcal)-\frac{a\deltabar\epss{\fcal}}{\rk{\fcal}} & \lleq\mu(\ecal)-\frac{a\deltabar}{\rk{\ecal}}\label{eq-def-slope-eps-ss}\\
& \mathbb{\kkvoid}\textbf{-(semi)stable } \quad & \hilpolred{\fcal}-\frac{\delta\kk{\fcal}{\ecal}}{\rk{\fcal}} & \ppreceq\hilpolred{\ecal}-\frac{a\delta}{\rk{\ecal}}\label{eq-def-k-ss}\\
& \textbf{slope }\mathbb{\kkvoid}\textbf{-(semi)stable } \quad & \mu(\fcal)-\frac{\deltabar\kk{\fcal}{\ecal}}{\rk{\fcal}} & \lleq\mu(\ecal)-\frac{a\deltabar}{\rk{\ecal}}\label{eq-def-slope-k-ss}
\end{align}
where $\hilpolred{\fcal}\doteqdot\frac{\hilpol{\fcal}}{\rk{\fcal}}$ is the reduced Hilbert polynomial.
\end{definition}

The above conditions are equivalent to the following:
   \begin{align*}
   (\ref{eq-def-eps-ss})\quad&\Leftrightarrow\quad \hilpol{\ecal}\rk{\fcal} - \rk{\ecal}\hilpol{\fcal} + a \delta (\rk{\ecal}\epss{\fcal}-\rk{\fcal})&\ssucceq0\\
   (\ref{eq-def-slope-eps-ss})\quad&\Leftrightarrow\quad \deg\ecal\rk{\fcal} - \rk{\ecal}\deg\fcal + a \deltabar (\rk{\ecal}\epss{\fcal}-\rk{\fcal})&\ggeq0\\
   (\ref{eq-def-k-ss})\quad&\Leftrightarrow\quad \hilpol{\ecal}\rk{\fcal} - \rk{\ecal}\hilpol{\fcal} + \delta (\rk{\ecal}\kk{\fcal}{\ecal}-a\rk{\fcal})&\ssucceq0\\
   (\ref{eq-def-slope-k-ss})\quad&\Leftrightarrow\quad \deg\ecal\rk{\fcal} - \rk{\ecal}\deg\fcal + \deltabar (\rk{\ecal}\kk{\fcal}{\ecal}-a\rk{\fcal})&\ggeq0.
   \end{align*}
Moreover note that $\kkvoid$-(semi)stability is equivalent to the usual (semi)stability applied to filtrations of length one. In fact let $\fcal$ be a subsheaf of $\ecal$ and consider the filtration $0\subset\fcal\subset\ecal$ with weight vector $\alphafil=\underline{1}$. An easy calculation shows that
\begin{align*}
 P(0\subset\fcal\subset\ecal,\underline{1})\,+\, & \delta\,\mu(0\subset\fcal\subset\ecal,\underline{1};\varphi)\,=\\
= & \,\hilpol{\ecal}\rk{\fcal}-\rk{\ecal}\hilpol{\fcal}\, +\, \delta\, (\rk{\ecal}\kk{\fcal}{\ecal}-a\rk{\fcal}).
\end{align*}

More precisely, these three notions of semistability are related in the following way:

\begin{proposition}\label{prop-eps-implica-ss-implica-k}
 $\eps$-(semi)stable $\Rightarrow$ (semi)stable $\Rightarrow$ $\kkvoid$-(semi)stable.
\end{proposition}
\begin{proof}
 Let $(\ecal,\varphi)$ be a $\eps$-(semi)stable decorated sheaf of rank $\rk{\ecal}=r$, let $(\filtrationcal)$ be a weighted filtration indexed by $\ttI$ and let $r_i$ be the rank of $\ecal_i$. For any $i\in\ttI$,
\begin{equation*}
 \hilpol{\ecal} r_i - r \hilpol{\ecal_i} + a \delta (r\epss{\ecal_i} - r_i)\ssucceq0,
\end{equation*}
therefore
\begin{equation*}
 \sum_{i\in\ttI}\alpha_i\left( \hilpol{\ecal}r_i - r\hilpol{\ecal_i} + a\delta(r\epss{\ecal_i} - r_i)\right)=\pttI+a\delta(r\sum_{i\in\ttI}\alpha_i\epss{\ecal_i} - \sum_{i\in\ttI}\alpha_ir_i)\ssucceq0.
\end{equation*} 
We want to show that $\pttI+\delta\muttI\ssucceq0$. Denote by $\eps_i\doteqdot\epss{\ecal_i}$ and let $j_0\doteqdot\min\{k\in\ttI\,|\,\eps_k\neq 0\}$.
Therefore
\begin{align*}
 \mu_\ttI & = -\min\{\gammattI^{(i_1)}+\dots+\gammattI^{(i_a)}\,|\,\rest{\varphi}{(E_{i_1}\otimes\dots\otimes E_{i_a})^{\oplus b}}\neq0\}\\
          & \geq -a \gammattI^{(j_0)}\\
          & = a\left(\sum_{i\geq j_0,i\in\ttI}\alpha_i r-\sum_{i\in\ttI}\alpha_i r_i\right)\\
          & = a\left(\sum_{i\in\ttI}\alpha_i \eps_i r-\sum_{i\in\ttI}\alpha_i r_i\right)\\
          & = a\left(\sum_{i\in\ttI} \alpha_i(\eps_i r-r_i)\right).
\end{align*}
So 
\begin{align*}
\pttI+\delta\muttI\ssucceq & \pttI+a\delta\left(\sum_{i\in\ttI} \alpha_i(\eps_i r-r_i)\right)=\\
= & \sum_{i\in\ttI}\alpha_i\left( \hilpol{\ecal}r_i - r\hilpol{\ecal_i} + a\delta(r\epss{\ecal_i} - r_i)\right)\ssucceq0, 
\end{align*}
and we are done.\\

Finally, given a (semi)stable decorated sheaf, we want to show that is $\kkvoid$-(semi)stable. Let $\fcal$ be a subsheaf of $\ecal$ of rank $r_{\fcal}$; if we consider the filtration $0\subset\fcal\subset\ecal$ with weights identically $1$, after small calculation ones get that
\begin{equation*}
0\ppreceq P(0\subset\fcal\subset\ecal,\underline{1})+\delta\mu(0\subset\fcal\subset\ecal,\underline{1};\varphi)=\hilpol{\ecal}r_{\fcal}-r\hilpol{\fcal}+\delta(r\kk{\fcal}{\ecal}-a r_{\fcal})
\end{equation*}
and we have done.
\end{proof}

Note that $\mu(\filtrationcal;\varphi)$ is not additive for all filtrations, i.e., it is not always true that 
\begin{equation}\label{eq-mu-additivity}
 \mu(\filtrationcal;\varphi)=\sum_{i\in\ttI}\mu(0\subset\ecal_i\subset\ecal,\alpha_i;\varphi).
\end{equation}

We will call \textbf{non-critical} a filtration for which (\ref{eq-mu-additivity}) holds and \textbf{critical} otherwise. Finally we will say that $\varphi$ is \textbf{additive} if equality (\ref{eq-mu-additivity}) holds for any weighted filtration, i.e., there are no critical filtrations.

\begin{remark}\label{rem-ss-uguale-kss-piu-condition-su-filtr-critiche}
\begin{enumerate}
 \item It easy to see that for any filtration (indexed by $\ttI$)
\begin{equation*}
\muttI(\filtrationcal;\varphi)\leq\sum_{i\in\ttI}\mu(0\subset\ecal_i\subset\ecal,\alpha_i;\varphi). 
\end{equation*}
 Therefore any subfiltration of a non-critical one is still non-critical. Indeed suppose that $\efilcal$ is a non critical filtration indexed by $\ttI$ and $\ttJ\subset\ttI$ indexes a critical subfiltration of $\efilcal$. Then $\muttI(\filtrationcal;\varphi)=\sum_{i\in\ttI}\mu(0\subset\ecal_i\subset\ecal,\alpha_i;\varphi)$ (since the whole filtration is non critical) and $\mu_{\ttJ}(\efilcal,\alphafil;\varphi)<\sum_{i\in\ttJ}\mu(0\subset\efilcal_i\subset\efilcal,\alpha_i;\varphi)$. Therefore $\mu_{\ttI\smallsetminus\ttJ}(\efilcal,\alphafil;\varphi)>\sum_{i\in\ttI\smallsetminus\ttJ}\mu(0\subset\ecal_i\subset\ecal,\alpha_i;\varphi)$ which is absurd.
 \item If $\varphi$ is additive $\kkvoid$-(semi)stability implies (semi)stability and therefore the two conditions are equivalent
 \item Checking semistability conditions over non-critical filtrations is the same to check them over subbundles.
 \item Thanks to the previous considerations, the following conditions are equivalent:
\begin{enumerate}
 \item $(\ecal,\varphi)$ is $\delta$-(semi)stable;
 \item For any subsheaf $\fcal$ and for any \underline{critical} filtration $(\filtrationcal)$ the following inequalities hold 
 \begin{align*}
0 & \ppreceq (\rk{\fcal}\hilpol{\ecal}-r\hilpol{\fcal})-\delta(r\kk{\fcal}{\ecal}-a\rk{\fcal}),\\
0 & \ppreceq P(\filtrationcal)+\delta\,\mu(\filtrationcal;\varphi).
 \end{align*}
\end{enumerate}
Observe that the first part of condition $(2)$ is just requiring that $(E,\varphi)$ is $\kkvoid$-(semi)stable.
 \item Note that Proposition \ref{prop-eps-implica-ss-implica-k} and points $2$, $3$ and $4$ above hold also for slope semistability.
\end{enumerate}
\end{remark}
\subsection{Decorated coherent sheaves}\label{sec-dec-coherent-sheaves}
With the expression ``\textbf{decorated coherent sheaf}'' we mean a decorated sheaf $(\acal,\varphi)$ such that $\acal$ is just a coherent sheaf (and not necessarily torsion free).\\
Before proceeding we recall what a decorated coherent \textit{subsheaf} is. If $i\colon(\fcal,\psi)\to(\acal,\varphi)$ is an injective morphism of decorated sheaves we get immediately from condition \eqref{eq-def-morphism-dec} that $\lambda\cdot\psi=i^*\varphi$. From now on we will say that the tripe $((\fcal,\psi),i)$ is a \textbf{decorated subsheaf} of $(\acal,\varphi)$ and we will denote it just by $(\fcal,\rest{\varphi}{\fcal})$. Note moreover that, if $\fcal$ is a subsheaf of $\acal$ and $i:\fcal\to\acal$ is the inclusion, then it defines a decorated subsheaf; in fact defining $\psi=\rest{\varphi}{i\ab(\fcal\ab)}$ the triple $((\fcal,\psi),i)$ is a decorated subsheaf of $(\acal,\varphi)$.\\

For an arbitrary decorated coherent sheaf $(\acal,\varphi)$ define the \textbf{\fr-decorated degree}
\begin{equation*}
\deg(\acal,\varphi)\doteqdot\deg(\acal)-a\deltabar \eps(\acal,\varphi),
\end{equation*}
where $\deg\acal\doteqdot c_1(\acal)\cdot\amplebu{1}^{n-1}$, and the \textbf{\fr-decorated Hilbert polynomial}
\begin{equation*}
\hilpol{(\acal,\varphi)}(m)\doteqdot \hilpol{\acal}(m)-a\delta(m)\eps(\acal,\varphi)
\end{equation*}
If moreover rank $\rk{\acal}>0$ we define the \textbf{\fr-decorated slope} and, respectively, the \textbf{reduced \fr-decorated Hilbert polynomial} as:
\begin{equation*}
 \mu(\acal,\varphi)\doteqdot\frac{\deg(\acal,\varphi)}{\rk{\acal}} \qquad\qquad \hilpolred{(\acal,\varphi)}(m)\doteqdot\frac{\hilpol{(\acal,\varphi)}(m)}{\rk{\acal}}.
\end{equation*}
Sometimes, if the morphism $\varphi$ is clear from the context, we will write $\ddeg(\acal)$ instead of $\deg(\acal,\varphi)$, $\dmu(\acal)$ instead of $\mu(\acal,\varphi)$ and $\hilpol{\acal}^\eps$ (respectively $\hilpolred{\acal}^\eps$) instead of $\hilpol{(\acal,\varphi)}$ (resp. $\hilpolred{(\acal,\varphi)}$).\\

\begin{definition}\label{def-eps-semistable}
Let $(\acal,\varphi)$ be a decorated coherent sheaf of positive rank, than we will say that $(\acal,\varphi)$ is \textbf{\fr-(semi)stable} or, respectively, \textbf{slope \fr-(semi)stable} with respect to $\delta$ (resp. $\deltabar$) if and only if for any proper non trivial subsheaf $\fcal\subset\acal$ the following inequality holds:
\begin{equation}\label{eq-def-eps-semista-dec-sheaf}
 \hilpol{(\fcal,\rest{\varphi}{\fcal})}\,\rk{\acal}\ppreceq\hilpol{(\acal,\varphi)}\,\rk{\fcal}.\\
\end{equation}
or, respectively,
\begin{equation}
\deg(\fcal,\varphi) \,\rk{\acal}\lleq\rk{\fcal} \, \deg(\acal,\varphi)
\end{equation}
If $\rk{\acal}=0$ we say that $(\acal,\varphi)$ is semistable (resp. slope semistable) or stable (resp. slope stable) if moreover $\hilpol{\acal}=\delta$ (resp. $\deg\acal=\deltabar$).
\end{definition}
In particular this \fr-(semi)stability extends $\eps$-semistability, defined in Section \ref{sec-others-notions-ss}, to decorated \textit{coherent} sheaves.

\begin{remark}
 Note that
\begin{center}
 slope \fr-stable $\Rightarrow$ \fr-stable $\Rightarrow$ \fr-semistable $\Rightarrow$ slope \fr-semistable;
\end{center}
and recall that \fr-semistability (slope \fr-semistability) is \textit{strictly stronger} than the usual semistability (resp. slope semistability) introduced in Section \ref{sec-semistability-conditions}.
\end{remark}

The kernel of $\varphi$ lies in $\acal\ab$, so for our purpose we need to define a subsheaf of $\acal$ that plays a similar role to the kernel of $\varphi$. Therefore we let 
\begin{equation*}
\kker\doteqdot\max\left\{0\subseteq\fcal\subseteq\acal \;|\; \fcal\ab\subseteq\ker\varphi\right\}
\end{equation*}
where the maximum is taken with respect to the partial ordering given by the inclusion of sheaves. Note that $\kker$ is unique, indeed if $\kker'$ is another maximal element, then $\kker\cup\kker'$ is a subsheaf of $\acal$ such that $(\kker\cup\kker')\ab\subset\ker\varphi$ and this is absurd.

\begin{remark}\label{rem-semistabilita}
\begin{enumerate}
 \item\label{c} Let $T(\acal)$ be the torsion part of $\acal$. The torsion part $T(\acal\ab)$ lies in $\ker\varphi$, otherwise there would be a non zero morphism between a sheaf of pure torsion and the torsion free sheaf $\linebu$ and this is impossible. Therefore also the twisted torsion part $T(\acal)\ab\subset T(\acal\ab)$ lies in the kernel of $\varphi$.
 \item $T(\acal)\subseteq\ker\varphi$ (point \eqref{c}) therefore $T(\acal)\subseteq\kker$;
 \item\label{b} $\acal$ is torsion free if and only if $\kker$ is torsion free. Indeed, suppose that $\kker$ is torsion free and that $T(\acal)\neq\emptyset$, then $\kker\subsetneq\kker\cup T(\acal)$ which is absurd for maximality of $\kker$. The converse is obvious.
 \item\label{a} $\acal\ab$ is torsion free if and only if $\ker\varphi$ is torsion free.
 \item If $(\acal,\varphi)$ is semistable and $\rk{\acal}>0$ then $\kker$ is a torsion free sheaf. Indeed if $T(\kker)$ is the torsion part of $\kker$, $\rk{T(\kker)}=0$ and than, for the semistability condition, we get that:
\begin{equation*}
\rk{\acal}\hilpol{(T(\kker)\ab,\rest{\varphi}{T(\kker)\ab})}\ppreceq 0.
\end{equation*}
Therefore $T(\kker)$ is zero and $\kker$ is torsion free.
\item If $(\acal,\varphi)$ is semistable and $\rk{\acal}>0$ then $\acal$ is pure of dimension $\dim X$ and therefore torsion free. Indeed let $\fcal$ a subsheaf of $\acal$ of pure torsion, then by the semistability condition we get that
\begin{equation*}
\rk{\acal}\left(\hilpol{\fcal}-a\delta\eps(\rest{\varphi}{\fcal})\right)\preceq\rk{\fcal}\left(\hilpol{\acal}-a\delta\right)=0.
\end{equation*}
Moreover, for point (1), $\fcal\ab\subset\ker\varphi$ and so $\hilpol{\fcal}\preceq0$, this immediately implies $\fcal=0$.
\end{enumerate}
\end{remark}

\begin{remark}
 Note that Remark \ref{rem-semistabilita} holds also in the slope \fr-semistable case.
\end{remark}

\begin{proposition}\label{prop-A-ss-sse-A/T(A)-ss}
 Let $(\acal,\varphi)$ be a decorated coherent sheaf of positive rank and let $T\doteqdot T(\acal)$ be the torsion of $\acal$. Then the following statements hold.
\begin{enumerate}
 \item $(\acal,\varphi)$ is \fr-semistable with respect to $\delta$ $\Longrightarrow$ $\acal/T$ is \fr-semistable with respect to $\delta$.
 \item $(\acal/T,\varphi)$ is \fr-semistable with respect to $\delta$ $\Longrightarrow$ $(\acal,\varphi)$ is \fr-semistable with respect to $\delta$ or $T$ is the maximal destabilizing subsheaf of $\acal$ in the sense of Remark \ref{rem-maximal-dest-sheaf-for-coherent}.
\end{enumerate}
\begin{proof}
 First of all note that, since $T\subset\ker\varphi$, the pair $(\acal/T,\varphi)$ is a well-defined decorated (torsion free) sheaf of the same type of $(\acal,\varphi)$.\\
\begin{enumerate}
 \item If $(\acal,\varphi)$ \fr-semistable with respect to $\delta$, then as in Remark \ref{rem-semistabilita} one can prove that $T=0$ and so obviously $(\acal/T,\varphi)$ is semistable.
 \item Suppose that $(\acal/T,\varphi)$ \fr-semistable with respect to $\delta$. If $T$ does not destabilize, then $\hilpol{T}^\eps=\hilpol{T}\preceq 0$ and so $T=0$ and $(\acal,\varphi)$ is \fr-semistable. Otherwise $\hilpol{T}^\eps=\hilpol{T}\succeq 0$ and Remark \ref{rem-maximal-dest-sheaf-for-coherent} shows that is the maximal destabilizing subsheaf.
\end{enumerate}
\end{proof}
\end{proposition}
\section[Mehta-Ramanathan for slope \fr-semistability]{Mehta-Ramanathan for slope $\eps$-semistable decorated sheaves}\label{sec-mehta-ramanathan}
In this section we want to prove a Mehta-Ramanathan theorem for slope \fr-semistable decorated sheaves. Before we proceed we need some notation and preliminary results.\\

\begin{notation}
 Let $k$ be an algebraic closed field of characteristic $0$, $S$ an integral $k$-scheme of finite type. $X$ will be a smooth projective variety over $k$, $\amplebu{1}$ an ample line bundle on $X$ and $f:X\to S$ a projective flat morphism. Note that $\amplebu{1}$ is also $f$-ample. In this section, unless otherwise stated, we will suppose that any decorated sheaf is of type $(a,b,\linebu)$. If $(\wcal,\varphi)$ is a decorated coherent sheaf over $X$ we denote by $\wcal_s$ the restriction $\rest{\wcal}{X_s}$, where $X_s\doteqdot f^{-1}(s)$, and $\varphi_s$ the restriction $\rest{\varphi}{\wcal_s}$. Finally, if $\fcal$ is a sheaf, we will denote by $\rkk{\fcal}$ the quantity $\rk{\fcal}$.\\
\end{notation}
\subsection{Maximal destabilizing subsheaf}\label{maximal-dest-subsheaf}
\begin{proposition}\label{prop-maximal-dest-subsheaf}
 Let $(\ecal,\varphi)$ be a decorated sheaf over a nonsingular projective smooth variety $X$. If $(\ecal,\varphi)$ is not \fr-semistable there is a unique, \fr-semistable, proper subsheaf $\fcal$ of $\ecal$ such that:
\begin{enumerate}
 \item $\hilpolred{\fcal}^\eps\succeq\hilpolred{\wcal}^\eps$ for all subsheaf $\wcal$ of $\ecal$.
 \item If $\hilpolred{\fcal}^\eps=\hilpolred{\wcal}^\eps$ then $\wcal\subset\fcal$.
\end{enumerate}
The subsheaf $\fcal$, with the induced morphism $\rest{\varphi}{\fcal}$, is called \textbf{maximal destabilizing subsheaf} of $(\ecal,\varphi)$.
\end{proposition}
\begin{proof}
 First we recall that by definition $\ecal$ is torsion free and therefore of positive rank.\\

We define a partial ordering on the set of decorated subsheaves of a given decorated sheaf $(\ecal,\varphi)$. Let $\fcal_1,\fcal_2$ two subsheaves of $\ecal$, then
\begin{equation}\label{eq-relaz-ordine}
 \fcal_1\preccurlyeq\fcal_2\quad\Longleftrightarrow\quad\fcal_1\subseteq\fcal_2\quad\wedge\quad\hilpol{(\fcal_1,\varphi_1)}\,\rk{\fcal_2}\,\preceq\,\hilpol{(\fcal_2,\varphi_2)}\,\rk{\fcal_1}\\
\end{equation}
where $\varphi_i=\rest{\varphi}{\fcal_i}$. Note that the set of the subsheaves of a sheaf $\ecal$ with this order relation $\preccurlyeq$ satisfies the hypothesis of Zorn's Lemma, so there exists a maximal element (non unique in general). Let
\begin{equation}
 \fcal\doteqdot\min_{\rk{\gcal}}\{\gcal\subset\ecal \,|\, \gcal \text{ is }\preccurlyeq\text{-maximal}\}
\end{equation}
i.e., $\fcal$ is a $\preccurlyeq$-maximal subsheaf with minimal rank among all $\preccurlyeq$-maximal subsheaves. Then we claim that $(\fcal,\rest{\varphi}{\fcal})$ has the asserted properties.\\

Suppose that exists $\gcal\subset\ecal$ such that 
\begin{equation}\label{eq-destab-max1}
 \hilpolred{\gcal}^\eps\succeq\hilpolred{\fcal}^\eps 
\end{equation}
First we show that we can assume $\gcal\subset\fcal$ by replacing $\gcal$ by $\fcal\cap\gcal$. Indeed if $\gcal\not\subset\fcal$, $\fcal$ is a proper subsheaf of $\fcal+\gcal$ in fact $\fcal\not\subset\gcal$ (otherwise $\fcal\preccurlyeq\gcal$ which is absurd for maximality of $\fcal$). By maximality one has that
\begin{equation}\label{eq-destab-max2}
 \hilpolred{\fcal}^\eps\succ\hilpolred{\fcal+\gcal}^\eps.
\end{equation}
 Using the exact sequence
\begin{equation*}
 0\longrightarrow\fcal\cap\gcal\longrightarrow\fcal\oplus\gcal\longrightarrow\fcal+\gcal\longrightarrow0
\end{equation*}
one finds $\hilpol{\fcal}+\hilpol{\gcal}=\hilpol{\fcal\oplus\gcal}=\hilpol{\fcal\cap\gcal}+\hilpol{\fcal+\gcal}$ and $\rk{\fcal}+\rk{\gcal}=\rk{\fcal\oplus\gcal}=\rk{\fcal\cap\gcal}+\rk{\fcal+\gcal}$. Hence
\begin{equation}\label{eq-proof-destab-max}
 \rkk{\fcal\cap\gcal}(\hilpolred{\gcal}-\hilpolred{\fcal\cap\gcal})=\rkk{\fcal+\gcal}(\hilpolred{\fcal+\gcal}-\hilpolred{\fcal})+(\rkk{\gcal}-\rkk{\fcal\cap\gcal})(\hilpolred{\fcal}-\hilpolred{\gcal}).
\end{equation}
where we denote by $\rkk{\fcal}$, $\rkk{\gcal}$, $\rkk{\fcal+\gcal}$ and $\rkk{\fcal\cap\gcal}$ the rank of $\rk{\fcal}$, $\rk{\gcal}$, $\rk{\fcal+\gcal}$ and $\rk{\fcal\cap\gcal}$ respectively.\\

If the morphism $\varphi$ is zero \fr-semistability coincides with the usual semistability for torsion-free sheaves and the existence of the maximal destibilizing subsheaf is a well-known fact that one can find, for example, in \cite{HLbook} Lemma 1.3.6. So we suppose that $\eps(\varphi)=1$. From the above inequalities between reduced decorated Hilbert polynomial of $\fcal$, $\gcal$ and $\fcal+\gcal$ one can easily obtain:
\begin{align*}
 & \hilpolred{\fcal+\gcal}-\hilpolred{\fcal}\prec a\delta\left(\frac{\epss{\fcal+\gcal}}{\rkk{\fcal+\gcal}}-\frac{\epss{\fcal}}{\rkk{\fcal}}\right)\\
 & \hilpolred{\fcal}-\hilpolred{\gcal}\preceq a\delta\left(\frac{\epss{\fcal}}{\rkk{\fcal}}-\frac{\epss{\gcal}}{\rkk{\gcal}}\right),
\end{align*}
therefore, using equation \eqref{eq-proof-destab-max} and after some easy computations, one gets
\begin{align*}
 \rkk{\fcal\cap\gcal} & (\hilpolred{\gcal}^\eps-\hilpolred{\fcal\cap\gcal}^\eps) = \\ 
 =\; 		& \rkk{\fcal+\gcal}(\hilpolred{\fcal+\gcal}-\hilpolred{\fcal})+(\rkk{\gcal}-\rkk{\fcal\cap\gcal})    (\hilpolred{\fcal}-\hilpolred{\gcal})-a\delta\epss{\gcal}\frac{\rkk{\fcal\cap\gcal}}{\rkk{\gcal}}+a\delta\epss{\fcal\cap\gcal}\\
 \prec\; 	& a\delta\rkk{\fcal\cap\gcal}\left(\frac{\epss{\fcal+\gcal}}{\rkk{\fcal+\gcal}}-\frac{\epss{\fcal}}{\rkk{\fcal}}\right)+a\delta(\rkk{\gcal}-\rkk{\fcal\cap\gcal})\left(\frac{\epss{\fcal}}{\rkk{\fcal}}-\frac{\epss{\gcal}}{\rkk{\gcal}}\right)-a\delta\epss{\gcal}\frac{\rkk{\fcal\cap\gcal}}{\rkk{\gcal}}+a\delta\epss{\fcal\cap\gcal}\\
 =\; 		& a\delta\left( \epss{\fcal+\gcal}-\epss{\fcal}-\epss{\gcal}+\epss{\fcal\cap\gcal}\right)\\
 \preceq\; 	& 0.
\end{align*}
Therefore we can suppose both $\gcal\subset\fcal$ and $\hilpolred{\gcal}^\eps\succeq\hilpolred{\fcal}^\eps$, and, up to replacing $\gcal$, we can suppose that $\gcal$ is maximal in $\fcal$ with respect to $\preccurlyeq$. Let $\gcal'$ be a $\preccurlyeq$-maximal in $\ecal$ among all subsheaves (of $\ecal$) containing $\gcal$. Then
\begin{equation*}
 \hilpolred{\fcal}^\eps\preceq\hilpolred{\gcal}^\eps\preceq\hilpolred{\gcal'}^\eps.
\end{equation*}
Note that neither $\gcal'$ is contained in $\fcal$, because $\fcal$ has minimal rank between all $\preccurlyeq$-maximal subsheaves of $\ecal$, nor $\fcal$ is contained in $\gcal'$, for maximality of $\fcal$; therefore $\fcal$ is a proper subsheaf of $\fcal+\gcal'$ and, for maximality, $\hilpolred{\fcal}^\eps\succeq\hilpolred{\fcal+\gcal'}^\eps$. As before one gets
\begin{equation*}
 \hilpolred{\fcal\cap\gcal'}^\eps\succ\hilpolred{\gcal'}^\eps\succeq\hilpolred{\gcal}^\eps,
\end{equation*}
but $\gcal\subseteq\fcal\cap\gcal'\subseteq\fcal$ and this contradicts the assumpions on $\gcal$.
Therefore $\fcal$ satisfies the required properties. The uniqueness and the \fr-semistability of $\fcal$ easily follow from properties $(1)$ and $(2)$.
\end{proof}

\begin{lemma}\label{lemma-maximal-dest-subsheaf}
Let $(\ecal,\varphi)$ be as before. If it is not slope \fr-semistable there is a unique proper subsheaf $\fcal$ of $\ecal$ such that:
\begin{enumerate}
 \item $\dmu(\fcal)\geq\dmu(\wcal)$ for all subsheaves $\wcal$ of $\ecal$.
 \item If $\dmu(\fcal)=\dmu(\wcal)$ then $\wcal\subset\fcal$.
\end{enumerate}
\end{lemma}
\begin{proof}
 The proof is the same of Proposition \ref{prop-maximal-dest-subsheaf}: it is sufficient to replace $\hilpolred{}^\eps$ with $\dmu$, $\hilpol{}^\eps$ with $\ddeg$ and $\delta$ with $\deltabar$.
\end{proof}

\begin{remark}
 Note that, if $(\ecal,\varphi)$ is \fr-semistable, or, respectively, slope \fr-semistable, then the maximal decorated destabilizing (resp. slope destabilizing) subsheaf coincides with $\ecal$.
\end{remark}

\begin{proposition}\label{rem-maximal-dest-sheaf-for-coherent}
 Let $(\acal,\varphi)$ be a decorated coherent sheaf of positive rank, then Proposition \ref{prop-maximal-dest-subsheaf} and Lemma \ref{lemma-maximal-dest-subsheaf} hold true, in the sense that if $(\acal,\varphi)$ is not \fr-semistable (slope \fr-semistable respectively) there is a unique, \fr-semistable, proper subsheaf $\fcal$ of $\ecal$ such that:
\begin{enumerate}
 \item $\hilpol{\fcal}^\eps\rk{\ecal}\succeq\hilpol{\wcal}^\eps\rk{\fcal}$ for all subsheaves $\wcal$ of $\ecal$.
 \item If $\hilpol{\fcal}^\eps\rk{\wcal}=\hilpol{\wcal}^\eps\rk{\fcal}$ then $\wcal\subset\fcal$.
\end{enumerate}
or, respectively
\begin{enumerate}
 \item[$1^{'}$.] $\ddeg(\fcal)\rk{\wcal}\geq\ddeg(\wcal)\rk{\fcal}$ for all subsheaves $\wcal$ of $\ecal$.
 \item[$2^{'}$.] If $\ddeg(\fcal)\rk{\wcal}=\ddeg(\wcal)\rk{\fcal}$) then $\wcal\subset\fcal$.
\end{enumerate}
\end{proposition}
\begin{proof}
 Indeed, let $\fcal$ be a minimal rank sheaf between all $\preccurlyeq$-maximal sheaves as in the proof of Proposition \ref{prop-maximal-dest-subsheaf}. Suppose that $\rk{\fcal}=0$, then $\fcal\subset\kker=T(\acal)$, and so, by maximality $\fcal=T(\acal)$. If $\gcal$ is such that $\rk{\gcal}>0$ and $T(\acal)\preccurlyeq\gcal$ then $\hilpol{T(\acal)}\preceq0$ but, by hypothesis, $T(\acal)$ destabilize and so $\hilpol{T(\acal)}\succ0$. Finally $T(\acal)$ is clearly unique and semistable.\\
 Otherwise, if $\rk{\fcal}>0$, then $\acal$ has no nontrivial rank zero subsheaves, in particular is torsion free. Indeed if exists a subsheaf $\gcal\subset\acal$ with $\rkk{\gcal}=0$ then, by the above considerations exists $\gcal'$ with $\rk{\gcal'}=0$, $\gcal\subseteq\gcal'$ and $\gcal'$ $\preccurlyeq$-maximal which is absurd by the assumptions on $\fcal$. Then the proof continues as the proof of Proposition \ref{prop-maximal-dest-subsheaf}.\\
The proof in the case of slope \fr-semistability is the same.
\end{proof}
\begin{remark}
 Proposition \ref{prop-maximal-dest-subsheaf} and Lemma \ref{lemma-maximal-dest-subsheaf} hold true also for decorated sheaves of type $(a,b,c,\linebu)$ as pointed out in $(1)$ Remark \ref{rem-diamond}.
\end{remark}
\subsection{Families of decorated sheaves}\label{sec-fam-of-decorated-sheaves}
Let $f:Y\to S$ be a morphism of finite type of Noetherian schemes. Recall that a \textbf{flat family} of coherent sheaves \textbf{on the fibre of the morphism $f$} is a coherent sheaf $\acal$ over $Y$, which is flat over $S$, i.e., for any $y\in Y$ $\acal_y$ is flat over the local ring $\ocal_{S,f(y)}$. If $\acal$ is flat over the fibre of $f$ the Hilbert polynomial $\hilpol{\acal_s}$ is locally constant as a function of $s$. The converse is not true in general, but, if $S$ is reduced, then the two assertions are equivalent.

\begin{definition}\label{def-flat-fam-over-fibre-morphism}
 Let $(\ecal,\varphi)$ be a \textit{decorated sheaf} over $Y$ of type $(a,b,\linebu)$ and $f:Y\to S$ be a morphism of finite type between Noetherian schemes. Then $(\ecal,\varphi)$ is a \textbf{flat family over the fibre of $f$} if and only if
\begin{itemize}
 \item[-] $\ecal$ and $\linebu$ are flat families of coherent sheaves over the fibre of $f:Y\to S$;
 \item[-] $\ecal_s=\rest{\ecal}{f^{-1}(s)}$ is torsion free for all $s\in S$;
 \item[-] $\linebu_s=\rest{\linebu}{f^{-1}(s)}$ is locally free for all $s\in S$;
 \item[-] $\eps(\varphi_s)=\eps(\rest{\varphi}{\ecal_s})$ is locally constant as a function of $s$.
\end{itemize}
 Note that the above conditions imply that the \fr-Hilbert polynomials $\hilpol{\ecal_s}^\eps$ are locally constant for $s\in S$.
\end{definition}

\begin{definition}
 Let $(\acal,\varphi)$ be a \textit{decorated coherent sheaf of positive rank}. Then $(\acal,\varphi)$ is a \textbf{flat family over the fibre of $f$} if and only if
\begin{itemize}
 \item[-] $\acal$ and $\linebu$ are flat families of coherent sheaves over the fibre of $f:Y\to S$;
 \item[-] $\linebu_s$ is locally free for all $s\in S$;
 \item[-] $\eps(\varphi_s)$ is locally constant as a function of $s$;
 \item[-] $\rk{\acal_s}>0$ for any $s\in S$.
\end{itemize}
 As in Definition \ref{def-flat-fam-over-fibre-morphism} the above conditions imply that the \fr-Hilbert polynomials $\hilpol{\acal_s}^\eps$ are locally constant for $s\in S$.
\end{definition}
\subsection{Families of quotients}\label{sec-families-quotients}
Let $(\acal,\varphi)$ be a decorated coherent sheaf over $X$ and let $q:\acal\twoheadrightarrow\qcal$ be surjective morphism of sheaves. Let $\fcal$ be the subsheaf of $\acal$ defined by $\ker q$, so the following succession of sheaves is exact:
\begin{equation*}
 0\longrightarrow\fcal\stackrel{i}{\longrightarrow}\acal\longrightarrow\qcal\longrightarrow0.
\end{equation*}
Note that $\fcal$ is uniquely determined by $\qcal$ and therefore also $(\fcal,\rest{\varphi}{\fcal})$ is uniquely (up to isomorphism of decorated sheaves) determined by $\qcal$. Indeed, let $(\fcal,\psi)$ be another decorated subsheaf of $(\acal,\varphi)$, then, by definition of decorated subsheaf, there exists a non-zero scalar morphism $\lambda:\linebu\to\linebu$ such that $\lambda\circ\psi=\varphi$, then
\begin{equation*}
 \xymatrix{
\fcal\ab \ar[rr]^{i\ab} \ar[dd]_{\psi} \ar[dr]^{\rest{\varphi}{\fcal}} &  & \acal\ab \ar[dd]^{\varphi}\\
 & \linebu \ar@{=}[dr]^{\text{id}} & \\
\linebu \ar[rr]_{\lambda} \ar@{.>}[ur]_{\lambda}& & \linebu.
}
\end{equation*}
Since the big square and the upper triangle commute, the entire diagram commutes and so it easy to see that $(\fcal,\psi)$ and $(\fcal,\rest{\varphi}{\fcal})$ are isomorphic as decorated sheaves.\\

Suppose now that $(\fcal,\rest{\varphi}{\fcal})$ de-semistabilizes a decorated sheaf $(\ecal,\varphi)$ (with respect to the slope \fr-semistability), then $\dmu(\fcal)>\dmu(\ecal)$ and so
\begin{align*}
 \mu(\fcal)-\frac{a\deltabar\epss{\fcal}}{\rkk{\fcal}} & > \mu(\ecal)-\frac{a\deltabar\epss{\ecal}}{\rkk{\ecal}}\\
 \deg(\fcal) & > \rkk{\fcal}\left[\mu(\ecal)+a\deltabar\left(\frac{\epss{\fcal}}{\rkk{\fcal}}-\frac{\epss{\ecal}}{\rkk{\ecal}}\right) \right]
\end{align*}
Recalling that $\deg(\ecal)=\deg(\fcal)+\deg(\qcal)$,
\begin{align*}
 \deg(\qcal) & < \deg(\ecal) - \rkk{\fcal}\left[\mu(\ecal)+a\deltabar\left(\frac{\epss{\fcal}}{\rkk{\fcal}}-\frac{\epss{\ecal}}{\rkk{\ecal}}\right) \right]\\
             & = \mu(\ecal)\rkk{\qcal}-a\deltabar\left(\epss{\fcal}-\epss{\ecal}\frac{\rkk{\fcal}}{\rkk{\ecal}}\right)
\end{align*}
and therefore, if $\epss{\ecal}=1$,
\begin{equation}\label{eq-equiv-destab-su-quoz-slope}
 \mu(\qcal) < \mu(\ecal)-a\deltabar\left(\frac{\epss{\fcal}}{\rkk{\qcal}}-\frac{\rkk{\fcal}}{\rkk{\qcal}\rkk{\ecal}}\right)=\mu(\ecal)+a\deltabar\cdot\begin{cases}
						    \frac{\rkk{\fcal}}{\rkk{\qcal}\rkk{\ecal}}\doteqdot C_0 \quad\text{ if } \epss{\fcal}=0\\
						    -\frac{1}{\rkk{\ecal}}\doteqdot -C_1\quad\text{ if } \epss{\fcal}=1,
						  \end{cases}
\end{equation}
otherwise, if $\epss{\ecal}=0$ then also $\epss{\fcal}=0$ and so we get that $\mu(\qcal)<\mu(\ecal)$.\\

\begin{remark}\label{rem-def-ddeg-al-quoziente}
 Defining $\ddeg(\qcal)\doteqdot\ddeg(\ecal)-\ddeg(\fcal)$ and $\dmu(\qcal)=\ddeg(\qcal)/\rkk{\qcal}$ one easily gets that $\dmu(\fcal)>\dmu(\ecal)$ if and only if $\dmu(\qcal)\leq\dmu(\ecal)$. Note that in general it is not possible to define a morphism $\psi$ over $\qcal$ such that $(\qcal,\psi)$ is a decorated sheaf and $\eps{(\qcal,\psi)}+\eps{(\fcal,\rest{\varphi}{\fcal})}=\eps{(\ecal,\varphi)}$. In fact it is possible to define a morphism to the quotient satisfying such properties if and only if $\kk{\fcal}{\ecal}=0$ or $\kk{\fcal}{\ecal}=a$. This is because only in these two cases the morphism $\overline{\varphi}$:
 \begin{equation*}
  \xymatrix{
  \fcal\ab \ar@{^{(}->}[r] & \ecal\ab \ar[r]^{\varphi} \ar@{->>}[d] & \linebu \\
   & \ecal\ab/\fcal\ab \ar[r]_{t} & (\ecal/\fcal)\ab \ar@{.>}[u]_{\overline{\varphi}}
  }
 \end{equation*}
 is well defined and so it is possible to give a well-defined structure of decorated sheaf to $(\ecal/\fcal)$.
\end{remark}

Analogously, if $(\fcal,\rest{\varphi}{\fcal})$ de-semistabilizes $(\ecal,\varphi)$ with respect to the \fr-semistability, i.e., if
\begin{equation*}
 \hilpolred{\fcal}-a\delta\frac{\epss{\fcal}}{\rkk{\fcal}} \succ \hilpolred{\ecal}-\frac{a\delta}{\rkk{\ecal}},
\end{equation*}
 then similar calculations show that
\begin{equation}\label{eq-equiv-destab-su-quoz-polinomio}
 \hilpolred{\qcal}\prec\hilpolred{\ecal}+a\delta\cdot\begin{cases}
						  C_0 \quad\text{ if } \epss{\fcal}=0\\
						  -C_1\quad\text{ if } \epss{\fcal}=1\\
						\end{cases}
\end{equation}
Note that condition \eqref{eq-equiv-destab-su-quoz-slope} implies condition \eqref{eq-equiv-destab-su-quoz-polinomio}, conversely, if $\hilpolred{\qcal}\prec\hilpolred{\ecal}+\delta C$ then $\mu(\qcal)\leq\mu(\ecal)+\deltabar C$.\\

Let $(\ecal,\varphi)$ be a flat family of decorated sheaves over the fibre of a projective morphism $f:X\to S$. Let $\hilpol{}=\hilpol{\ecal_s}$ and $\hilpolred{}=\hilpolred{\ecal_s}$ the Hilbert polynomial and, respectively, the reduced Hilbert polynomial of $\ecal$ (which are constant because the family is flat over $S$). Define:
\begin{enumerate}
 \item $\mathfrak{F}$ as the family (over the fibre of $f$) of saturated subsheaves $\fcal\hookrightarrow\ecal_s$ such that the induced torsion free quotient $\ecal_s\twoheadrightarrow\qcal$ satisfy $\mu(\qcal)\leq\mu(\ecal_s)+a\deltabar C_0$;
 \item $\mathfrak{F_0}$ as the family of decorated subsheaves $(\fcal,\rest{\varphi}{\fcal})\hookrightarrow(\ecal_s,\rest{\varphi}{\ecal_s})$ such that:\begin{itemize}
	 \item[-] $\eps{(\fcal,\rest{\varphi}{\fcal})}=0$;
	 \item[-] $\hilpolred{\fcal}^\eps=\hilpolred{\fcal}\succ\hilpolred{\ecal}^\eps$ (i.e., $\hilpolred{\qcal}\prec\hilpolred{}+a\delta C_0$ with $\qcal\doteqdot\text{coker}(\fcal\hookrightarrow\ecal_s)$);
	 \item[-] $\fcal$ is a saturated subsheaf of $\ecal_s$;
      \end{itemize}
 \item $\mathfrak{F_1}$ as the family of decorated subsheaves $(\fcal,\rest{\varphi}{\fcal})\hookrightarrow(\ecal_s,\rest{\varphi}{\ecal_s})$ such that:\begin{itemize}
	 \item[-] $\eps{(\fcal,\rest{\varphi}{\fcal})}=1$;
	 \item[-] $\hilpolred{\fcal}^\eps\succ\hilpolred{\ecal}^\eps$  (i.e., $\hilpolred{\qcal}\prec\hilpolred{}-a\delta C_1$ with $\qcal\doteqdot\text{coker}(\fcal\hookrightarrow\ecal_s)$);
	 \item[-] $\fcal$ is a saturated subsheaf of $\ecal_s$;
      \end{itemize}
\end{enumerate}

We want to prove that the set of Hilbert polynomials of destabilizing decorated subsheaves of a flat family $(\ecal,\varphi)$ of decorated sheaves over the fibre of a projective morphism $f:X\to S$ is a finite set. From this we conclude that the semistability condition is an open condition, i.e., the set $\{s\in S \,|\, (\ecal_s,\varphi_s) \text{ is slope \fr-semistable}\}$ is open in $S$. In order to prove this result we first need to recall some facts.

\begin{definition}\label{def-bounded-family}
 A family of isomorphism classes of coherent sheaves on a projective scheme $Y$ over $k$ is \textbf{bounded} if there is a $k$-scheme $S$ of finite type and a coherent $\ocal_{S\times Y}$-sheaf $\gcal$ such that the given family is contained in the following set: $\{ \rest{\gcal}{\text{Spec}(k(s))\times Y} \,|\, s \text{ is a closed point in } S\}$.
\end{definition}

\begin{definition}\label{de-m-regular}
A sheaf $\acal$ over $Y$ is said \textbf{$m$-regular} if
\begin{equation*}
 H^i(Y,\acal(m-i))=0 \text{ for all }i>0.
\end{equation*}
Define the \textbf{Mumford-Castelnuovo regularity} of $\acal$ as 
\begin{equation*}
 \text{reg}(\acal)\doteqdot\inf\{m\in\zz \,|\, \acal \text{ is }m\text{-regular}\}
\end{equation*}
\end{definition}

Then the following statements hold:

\begin{lemma}[Lemma 1.7.2 \cite{HLbook}]\label{lemma-equiv-m-regular}
 If $\acal$ is $m$-regular, then
 \begin{enumerate}
  \item[i)] $\acal$ is $m'$-regular for all integers $m'\geq m$.
  \item[ii)] $\acal(m)=\acal\otimes\amplebu{m}$ is globally generated.
  \item[iii)] For all $n\geq0$ the natural homomorphisms
  \begin{equation*}
   H^0(X,\acal(m))\otimes H^0(X,\amplebu{n})\longrightarrow H^0(X,\acal(m+n))
  \end{equation*}
  are surjective.
 \end{enumerate}

\end{lemma}

\begin{lemma}[Lemma 1.7.6 of \cite{HLbook}]\label{lemma-equiv-boundness-conditions}
 The following properties of families of sheaves $\{\acal_i\}_{i\in I}$ are equivalent:
\begin{enumerate}
 \item[i)] the family is bounded;
 \item[ii)] the set of Hilbert polynomials $\{\hilpol{\acal_i}\}_{i\in I}$ is finite and there is a uniform bound for $\text{reg}(\acal_i)\leq C$ for all $i\in I$
 \item[iii)] the set of Hilbert polynomials $\{\hilpol{\acal_i}\}_{i\in I}$ is finite and there is a coherent sheaf $\acal$ such that all $\acal_i$ admit surjective morphisms $\acal\to\acal_i$.
\end{enumerate}
\end{lemma}

\begin{definition}
 Let $\acal$ be a coherent sheaf. We call \textbf{hat-slope} the rational number
\begin{equation*}
 \hat{\mu}(\acal)\doteqdot\frac{\beta_{\dim \acal-1}(\acal)}{\beta_{\dim \acal}(\acal)},
\end{equation*}
where $\beta_i(\acal)$ is defined as the coefficient of $x^i$ of the Hilbert polynomial of $\acal$ multiplied by $i!$, i.e., if $\hilpol{\ecal}(x)=\sum_{i=0}^{\dim\acal}\beta_i \frac{x^i}{i!}$ then $\beta_i(\acal)\doteqdot\beta_i$.
\end{definition}

\begin{lemma}[Lemma $2.5$ in \cite{Grothendieck1}]
Let $f:Y\to S$  be a projective morphism of Noetherian schemes and denote by $\ocal_Y(1)$ a line bundle on $Y$, which is very ample relative to $S$. Let $\acal$ be a coherent sheaf on $Y$ and $\mathfrak{Q}$ the set of isomorphism classes of quotients sheaves $\qcal$ of $\acal_s$ for $s$ running over the points of $S$. Suppose that the dimension of $Y_s$ is $\leq r$ for all $s$. Then the coefficient $\beta_r(\qcal)$ is bounded from above and below, and $\beta_{r-1}(\qcal)$ is bounded from below. If $\beta_{r-1}(\qcal)$ is bounded from above, then the family of sheaves $\qcal/T(\qcal)$ is bounded.
\end{lemma}

\begin{proposition}\label{prop-limitatezza-fam-quozienti-slope-buonded}
 Let $\acal$ be a flat family of coherent sheaves on the fibres of a projective morphism $f:Y\to S$ of Noetherian schemes. Then the family of torsion free quotient $\qcal$ of $\acal_s$ for $s\in S$ with hat slope bounded from above is a bounded family.
\end{proposition}
\begin{proof}
 It is an easy corollary of Lemma 2.5 in \cite{Grothendieck1}.
\end{proof}

Thanks to Proposition \ref{prop-limitatezza-fam-quozienti-slope-buonded} the family $\mathfrak{F}$ is bounded. Due the previous considerations both families $\mathfrak{F_0}$ and $\mathfrak{F_1}$ can be regarded as subfamilies of $\mathfrak{F}$ and therefore $\mathfrak{F_0}$ and $\mathfrak{F_1}$ are bounded families as well. Thanks to Proposition \ref{lemma-equiv-boundness-conditions} the sets $\{\hilpol{\fcal} \,|\, \fcal\in\mathfrak{F_0}\}$ and $\{\hilpol{\fcal} \,|\, \fcal\in\mathfrak{F_1}\}$ are finite.
\subsection{Quot schemes}\label{sec-quot-schemes}
Let $\acal$ be a coherent sheaf over $X$ flat over the fibres of $f:X\to S$. Let $P\in\qq[x]$ be a polynomial. Define a functor 
\begin{equation*}
\qttfunctor\doteqdot\quotfunctor_{X/S}(\acal,P):(\text{Sch}/S)\to (\text{Sets})
\end{equation*}
as follows: if $T\to S$ is scheme over $S$ let $\qttfunctor(T)$ be the set of all $T$-flat coherent quotient sheaves $\acal_T\twoheadrightarrow\qcal$ with Hilbert polynomial $P$, where $\acal_T$ denotes the sheaf over $X_T=X\times_{S}T$ induced by $\acal$. If $g:T'\to T$ is an $S$-morphism, let $\qttfunctor(g):\qttfunctor(T)\to\qttfunctor(T')$ be the map that sends $\acal_T\twoheadrightarrow\qcal$ to $\acal_{T'}\twoheadrightarrow g_X^*\qcal$, where $g_X:X_{T'}\to X_T$ is the map induced by $g$.
\begin{theorem}[Theorem 2.2.4 in \cite{HLbook}]\label{teo-quot-scheme}
The functor $\quotfunctor_{X/S}(\acal,P)$ is represented by a projective $S$-scheme $\pi:\quot_{X/S}(\acal,P)\to S$.
\end{theorem}

Consider now a decorated coherent sheaf $(\acal,\varphi)$ over $X$, flat over the fibre of $f:X\to S$ and let $P\in\qq[x]$ be a polynomial. Define the functor
\begin{equation*}
\qttfunctor^\mathfrak{0}\doteqdot\quotfunctor^\mathfrak{0}_{X/S}(\acal,\varphi,P):(\text{Sch}/S)\to (\text{Sets})
\end{equation*}
as follows: if $T\to S$ is scheme over $S$ let $\qttfunctor^\mathfrak{0}(T)$ be the set of all $T$-flat coherent quotient sheaves $\acal_T\twoheadrightarrow\qcal$ with Hilbert polynomial $P$ such that $\eps(\rest{\varphi_T}{\ker(\acal_T\twoheadrightarrow\qcal)})=0$, where $\acal_T$ denotes the sheaf over $X_T=X\times_{S}T$ induced by $\acal$ and $\varphi_T:(\acal_T)\ab\to\acal\ab\stackrel{\varphi}{\longrightarrow}\linebu$ is the morphism induced by $\varphi$. If $g:T'\to T$ is an $S$ morphism, let $\qttfunctor^\mathfrak{0}(g):\qttfunctor^\mathfrak{0}(T)\to\qttfunctor^\mathfrak{0}(T')$ be the map that sends $\acal_T\twoheadrightarrow\qcal$ to $\acal_{T'}\twoheadrightarrow g_X^*\qcal$, note that $g_X^*\varphi_T$ is zero if restricted on $\ker(\acal_{T'}\twoheadrightarrow g_X^*\qcal)$.
\begin{theorem}
The functor $\quotfunctor^\mathfrak{0}_{X/S}(\acal,\varphi,P)$ is represented by a projective $S$-scheme $\pi^\mathfrak{0}:\quot^\mathfrak{0}_{X/S}(\acal,\varphi,P)\to S$ that is a closed subscheme of $\quot_{X/S}(\acal,P)$.
\end{theorem}
\begin{proof}
 The additional property is closed and therefore, using the same arguments of the proof of Theorem $1.6$ in \cite{Sernesi}, one can prove that $\quot^\mathfrak{0}_{X/S}(\acal,\varphi,P)=\{q\in\quot_{X/S}(\acal,P) \,|\, \eps(\rest{\varphi}{\ker(q)})=0\}$ is a closed projective subscheme of $\quot_{X/S}(\acal,P)$.
\end{proof}
\subsection{Openness of semistability condition}\label{sec-opennes-semistab-condition}
\begin{proposition}\label{prop-eps-semistability-open-condition}
 Let $f:X\to S$ be a projective morphism of Noetherian schemes and let $(\ecal,\varphi)$ be a flat family of decorated sheaves over the fibre of $f$. The set of points $s\in S$ such that $(\ecal_s,\varphi_s)$ is \fr-(semi)stable with respect to $\delta$ is open in $S$.
\end{proposition}
\begin{proof}
 Let $\hilpol{}=\hilpol{\ecal_s}$ and $\hilpolred{}=\hilpolred{\ecal_s}$ the Hilbert polynomial and, respectively, the reduced Hilbert polynomial of $\ecal$. 
We first consider the semistable case. Let
\begin{equation}
 A\doteqdot\{P''\in\qq[x] \,|\, \exists\, s\in S, \exists\, q:\ecal_s\twoheadrightarrow\qcal \text{ such that }\hilpol{\qcal}=P'' \text{ and } \ker(q)\in\mathfrak{F}\}
\end{equation}
and, for $i=0,1$, let
\begin{equation*}
 A_i\doteqdot\{P''\in\qq[x] \,|\, \exists\, s\in S, \exists\, q:\ecal_s\twoheadrightarrow\qcal \text{ such that }\hilpol{\qcal}=P'' \text{ and } \ker(q)\in\mathfrak{F_i}\}
\end{equation*}
The sets $A$, $A_0$ and $A_1$ are finite because the families $\mathfrak{F}$, $\mathfrak{F_0}$ and $\mathfrak{F_1}$ are bounded as proved in Section \ref{sec-families-quotients}. For any $P''\in A_1$ consider the Quot scheme $\pi:\quotfunctor_{X/S}(\ecal,P'')\to S$, while for $P''\in A_0$ consider the Quot scheme $\pi^\mathfrak{0}:\quotfunctor^\mathfrak{0}_{X/S}(\ecal,\varphi,P'')\to S$. Both images $S(P'')$ of $\pi$ (for $P''\in A_1$) and $S^{\mathfrak{0}}(P'')$ of $\pi^{\mathfrak{0}}$ (for $P''\in A_0$) are closed sets of $S$. Therefore the union
\begin{equation*}
 \left(\bigcup_{P''\in A_0} S^{\mathfrak{0}}(P'')\right) \cup \left(\bigcup_{P''\in A_1} S(P'')\right)
\end{equation*}
is a closed subset of $S$, in fact it is finite union of closed sets. Finally is easy to see that $(\ecal_s,\varphi_s)$ is semistable if and only if $s$ is \textit{not} in the above union.\\
The proof of the stable case is similar to the semistable case, it is indeed sufficient to consider, for $i=0,1$, the sets
\begin{align*}
 A^{\text{st}}_i\doteqdot\{P''\in A \,|\, \text{ with } \hilpolred{\qcal}\preceq\hilpolred{}+(1-i)(-a\delta C_0)+i(a\delta C_1)\}
\end{align*}
and continue as in the semistable case.
\end{proof}
\subsection{Relative maximal destabilizing subsheaf}\label{sec-relative-eps-maximal-dest-subsheaf}
\begin{theorem}\label{teo-rel-max-dest}
 Let $(X,\amplebu{1})$, $S$, $f:X\to S$ and $(\ecal,\varphi)$ as before. Then there is an integral $k$-scheme $T$ of finite type, a projective birational morphism $g:T\to S$, a dense open subset $U\subset T$ and a flat quotient $\qcal$ of $\ecal_T$ such that for all points $t\in U$, $\fcal_t\doteqdot\ker(\ecal_t\twoheadrightarrow\qcal_t)$ with the induced morphism $\rest{\varphi_t}{\fcal_t}$ is the maximal destabilizing subsheaf of $(\ecal_t,\varphi_t)$ or $\qcal_t=\ecal_t$.\\
 Moreover the pair $(g,\qcal)$ is universal in the sense that if $g':T'\to S$ is any dominant morphism of $k$-integral schemes and $\qcal'$ is a flat quotient of $\ecal_{T'}$, satisfying the same property of $\qcal$, there is an $S$-morphism $h:T'\to T$ such that $h^*_X(\qcal)=\qcal'$.
\end{theorem}
\begin{proof}
In the proof we apply the same arguments as in \cite{Frank-rest}. Define $B_1=A_1$ and $B_0=A_0'$, i.e.,
\begin{align*}
& B_0=\{P''\in A \,|\, \hilpolred{\qcal}\preceq\hilpolred{}-a\delta C_0\}\\
& B_1=\{P''\in A \,|\, \hilpolred{\qcal}\prec\hilpolred{}+a\delta C_1\}
\end{align*}
Then define
\begin{align*}
 & \check{B}_0\doteqdot\{ P''\in B_0 \,|\, \pi^{\mathfrak{0}}(\quot^{\mathfrak{0}}_{X/S}(\ecal,\varphi,P''))=S\}\\
 & \check{B}_1\doteqdot\{ P''\in B_1 \,|\, \pi(\quot_{X/S}(\ecal,P''))=S \text{ and } \;\forall s\in S\;\; \pi^{-1}(s)\not\subset\quot^{\mathfrak{0}}_{X/S}(\ecal,\varphi,P''))\}
\end{align*}
Note that $B_0\cup B_1$ and $\check{B}_0\cup\check{B}_1$ are nonempty. We want to define an order relation on $B_0,\check{B}_0,B_1$ and $\check{B}_1$ but first we need the following costruction: let $P_1'',P_2''$ be polynomials in $B_0,\check{B}_0,B_1$ or $\check{B}_1$; then there exist surjective morphisms $q_i:\ecal_s\to \qcal_i$ ($i=1,2$) such that $P_i''=\hilpol{\qcal_i}$. Define, for $i=1,2$, $P_i\doteqdot\hilpol{\ker(q_i)}$, $r_i\doteqdot\rk{\ker(q_i)}$ and $p_i=P_i/r_i$. We will say that the polynomials $P_i$ are \textit{associated} with the polynomials $P_i''$.\\

If $P_i''\in B_0$ or $\check{B}_0$ define the following ordering relation:
\begin{equation*}
 P_1''\vartriangleleft P_2'' \Longleftrightarrow p_1\succ p_2 \quad\text{or}\quad p_1=p_2 \text{ and } r_1>r_2,
\end{equation*}
otherwise, if $P_i''\in B_1$ or $\check{B}_1$, define:
\begin{equation*}
 P_1''\vartriangleleft P_2'' \Longleftrightarrow  p_1 - \frac{a\delta}{r_1}\succ p_2 - \frac{a\delta}{r_2} \quad\text{or}\quad p_1 - \frac{a\delta}{r_1} = p_2 - \frac{a\delta}{r_2}\text{ and } r_1>r_2
\end{equation*}
Let $P_-^{''i}$, for $i=0,1$, be a $\vartriangleleft$-minimal polynomial among all polynomials in $\check{B}_i$ and $P_-^i$ the associated polynomials. Then consider the following cases:
\begin{itemize}
 \item[Case $1$:] $p_-^0\succ p_-^1-\frac{a\delta}{r_-^1}$;
 \item[Case $2$:] $p_-^0\prec p_-^1-\frac{a\delta}{r_-^1}$;
 \item[Case $3$:] $p_-^0 = p_-^1-\frac{a\delta}{r_-^1}$ and $r_-^0 > r_-^1$;
 \item[Case $4$:] $p_-^0 = p_-^1-\frac{a\delta}{r_-^1}$ and $r_-^0 < r_-^1$;
\end{itemize}
In the first and third case define $P_-^{''}=P_-^{''0}$, in the second and fourth case put $P_-^{''}=P_-^{''1}$. Note that the set
\begin{equation*}
 U_-^\prime\doteqdot\left( \bigcup_{P''\in B_0, P''\vartriangleleft P_-^{''0}}  \pi^{\mathfrak{0}}(\quot^{\mathfrak{0}}_{X/S}(\ecal,\varphi,P'')) \right) \cup  \left( \bigcup_{P''\in B_1, P''\vartriangleleft P_-^{''1}}  \pi(\quot_{X/S}(\ecal,P'')) \right)
\end{equation*}
is a proper closed subscheme of $S$. In fact it is proper and closed because it is a finite union of closed proper subschemes of $S$. Call $U_-$ its complement in $S$.\\

Suppose that $P_-^{''}\in\check{B}_0$. By definition the projective morphism
\begin{equation*}
 \pi^{\mathfrak{0}}(\quot^{\mathfrak{0}}_{X/S}(\ecal,\varphi,P_-''))\to S
\end{equation*}
is surjective and for any point $s\in S$ the fibre of $\pi^{\mathfrak{0}}$ at $s$ parametrizes possible quotients with Hilbert polynomial $P_-^{''}$. The associated subsheaf of any such quotient is, by costruction, the maximal decorated destabilizing subsheaf. The case that $P_-^{''}\in\check{B}_1$ is similar. Finally by re-adapting the techniques used in the proof of the corresponding result in \cite{Frank-rest}, one concludes.
\end{proof}
\subsection{Restriction theorem}\label{sec-restriction-theorem}
Let $X$ be a smooth projective variety and $\amplebu{1}$ be a fixed ample line bundle. Let $(\ecal,\varphi)$ be a decorated sheaf of type $(a,b,\linebu)$ over $X$ with non-zero decoration morphism.
For a fixed positive integer $\ai\in\nn^+$, we define:
\begin{itemize}
 \item[-] $\Pi_\ai\doteqdot|\amplebu{\ai}|$ the complete linear system of degree $\ai$ in $X$;
 \item[-] $Z_\ai\doteqdot\{(D,x)\in\Pi_\ai\times X \, | \, x\in D \}$ the incidence variety with projections
\begin{equation*}
\xymatrix{\Pi_\ai\times X & Z_\ai \ar@{_{(}->}[l]\ar[r]^{q_\ai} \ar[d]_{p_\ai} & X\\
& \Pi_\ai
}
\end{equation*}
\end{itemize}
One can prove (see Section $2$ of \cite{Mehta-Ram}) that:
\begin{equation}\label{eq-pic}
 \pic(Z_\ai)=q_\ai^*\pic(X)\oplus p_\ai^*\pic(\Pi_\ai).
\end{equation}

For any sheaf $\gcal$ over $X$ one has $\rest{P_{\gcal}}{D}(n)=P_{\gcal}(n)-P_{\gcal}(n-\ai)$, therefore, given a decorated sheaf $(\ecal,\varphi)$ over $X$ with decoration of type $\tipo=(a,b,\linebu)$, for all $D\in\Pi_\ai$ the restrictions $\rest{\ecal}{D}$ and $\rest{\linebu}{D}$ have constant Hilbert polynomials. Since $\Pi_\ai$ is reduced, as remarked at the beginning of Section \ref{sec-fam-of-decorated-sheaves}, it follows that $q_{\ai}^*\linebu$ and $q_{\ai}^*\ecal$ are flat families of sheaves on the fibre of $p_\ai:Z_\ai\to\Pi_\ai$.\\

\begin{remark}
 If for any $D\in\Pi_\ai$ ${\rest{\varphi_\ai}{\rest{(q_\ai^*\ecal)}{p_\ai^{-1}(D)}}}=\rest{\varphi}{\left(\rest{\ecal}{D}\right)}\neq0$, the family of decorated sheaves $(q_{\ai}^*\ecal,q_{\ai}^*\varphi)$ is flat. Otherwise, since to be nonzero is open condition, there exists a dense open subset of $\Pi_\ai$ over which $(q_{\ai}^*\ecal,q_{\ai}^*\varphi)$ is flat.
\end{remark}

Thanks to this remark and Theorem \ref{teo-rel-max-dest}, there exist a dense open subset $\vai$ of $\Pi_\ai$ and a torsion-free sheaf $\qcal_\ai$ over $Z_\vai\doteqdot Z_\ai\times_{\Pi_\ai}\vai$ such that:
\begin{itemize}
 \item $(\ecal_\ai,\varphi_\ai)\doteqdot(q_{\ai}^*\ecal,q_{\ai}^*\varphi)$ is flat over $\vai$;
 \item $\qcal_\ai$ is flat over $\vai$;
 \item $\fcal_\ai\doteqdot\ker(\ecal_\ai\to\qcal_\ai)$, with the induced morphism $\rest{\varphi_\ai}{\fcal_\ai}$, is the relative maximal decorated destabilizing subsheaf of $(\ecal_\ai,\varphi_\ai)$; i.e., for any $D\in\vai$ $\rest{\fcal_\ai}{p_\ai^{-1}(D)}$ (with the induced morphism) de-semistabilize $\rest{(\ecal_\ai,\varphi_\ai)}{p_\ai^{-1}(D)}$.
 \end{itemize}
 Recall that:
 \begin{itemize}
 \item by construction of the relative maximal decorated destabilizing subsheaf, the quantity
\begin{equation*}
 \eps\left( \rest{\fcal_\ai}{p_\ai^{-1}(D)},\rest{\varphi_\ai}{\left(\rest{\fcal_\ai}{p_\ai^{-1}(D)}\right)} \right)
\end{equation*}
depends only on $\ai$ and not on $D\in\vai$ and for this reason from now on we will denote by $\eps(\ai)$;
 \item $(\ecal_\ai,\varphi_\ai)$, $(\fcal_\ai,\rest{\varphi_\ai}{\fcal_\ai})$ and $\qcal_\ai$ are flat families of decorated sheaves (resp. sheaves) over $\vai$.
\end{itemize}

Let $\deter$ be a line bundle which extends $\det(\qcal_\ai)$ to all $Z_\ai$; in view of \eqref{eq-pic} the line bundle $\deter$ can be uniquely decomposed as $\deter=q_\ai^*L_\ai\otimes p_\ai^* M_\ai=L_\ai\boxtimes M_\ai$ with $L_\ai\in\pic(X)$ and $M_\ai\in\pic(\Pi_\ai)$. Note that $\deg(\rest{\qcal_\ai}{p_\ai^{-1}(D)})=\ai\deg(L_\ai)$.\\

For a general divisor $D\in\Pi_\ai=|\amplebu{\ai}|$, let $\deg(\ai)$, $\rk{\ai}$ and $\mu(\ai)$ denote the degree, rank and slope of the maximal decorated destabilizing subsheaf $\rest{(\fcal_\ai,\rest{\varphi_\ai}{\fcal_\ai})}{p_\ai^{-1}D}$ of $(\rest{\ecal_\ai}{p_\ai^{-1}D},\rest{\varphi_\ai}{p_\ai^{-1}D})$. Let $\dmu(\ai)=\mu(\ai)-\frac{a\deltabar\eps(\ai)}{\rk{\ai}}$, $\degq(\ai)=\deg(\rest{\ecal_\ai}{p_\ai^{-1}D})-\deg(\ai)$, $\rkq{\ai}=\rk{\rest{\ecal_\ai}{p_\ai^{-1}D}}-\rk{\ai}$ and $\eps^q(\ai)=\epss{\rest{\ecal_\ai}{p_\ai^{-1}D}}-\eps(\ai)$. Finally $\muq(\ai)=\frac{\degq(\ai)}{\rkq{\ai}}$, $\ddegq(\ai)=\deg(\ai)-a\deltabar\eps^q(\ai)$ and $\dmuq(\ai)=\muq(\ai)-a\delta\frac{\eps^q(\ai)}{\rkq{\ai}}=\frac{\ddegq(\ai)}{\rkq{\ai}}$.\\

Let $U_\ai\subset\vai$ denote the dense open set of points $D\in V_\ai$ such that $D$ is smooth.\\

\begin{lemma}[Lemma 7.2.3 in \cite{HLbook}]\label{lem-esistono-D_i}
 Let $\ai_1,\dots,\ai_l$ be positive integers, $\ai=\sum_i\ai_i$ and $D_i\in U_{\ai_i}$ divisors such that $D=\sum_i D_{\ai_i}$ is a divisor with normal crossing. Then there is a smooth locally closed curve $C\subset\Pi_\ai$ containing the point $D$ such that $C\smallsetminus\{D\}\subset U_\ai$ and $Z_C\doteqdot C\times_{\Pi_\ai}Z_\ai$ is smooth in codimension $2$. 
\end{lemma}

\begin{lemma}
 Let $\ai_1,\dots,\ai_l$ be positive integers and $\ai=\sum_i\ai_i$. Then
 \begin{itemize}
  \item $\mu(\ai)\leq\sum_i\mu(\ai_i)$,
  \item $\muq(\ai)\geq\sum_i\muq(\ai_i)$,
  \item $\dmuq(\ai)\geq\sum_i\dmuq(\ai_i)$
 \end{itemize}
 and in case of equality $\rkq{\ai}\leq\min_i\rkq{\ai_i}$, or equivalently $\rk{\ai}\geq\max_i\rk{\ai_i}$.
\end{lemma}
\begin{proof}
 Let $D_i\in U_{\ai_i}$, for $i=1,\dots,l$, be divisors satisfying the requirements of Lemma \ref{lem-esistono-D_i}, be $D\doteqdot\sum_i D_i$ and let $C$ be a curve with the properties of Lemma \ref{lem-esistono-D_i}. There exists over $\vai$ a maximal decorated destabilizing subsheaf $\fcal_\ai$ with the associated torsion free quotient $\rest{\ecal_\ai}{Z_\vai}\to\qcal_\ai$. Recall that both sheaves are flat over $\vai$. Its restriction to $\vai\cap C$ can uniquely be extended to a $C$ flat quotient $\rest{\ecal_\ai}{Z_C}\to\qcal_C$ and let $\fcal_C=\ker(\rest{\ecal_\ai}{Z_C}\to\qcal_C)$, then also $\fcal_C$ extends $\rest{\fcal_\ai}{\vai\cap C}$ to all $C$. Note that also $\fcal_C$ is flat over $C$ and so $\hilpol{\rest{\fcal_C}{D}}=\hilpol{\rest{\fcal_C}{c}}$ for any $c\in C$. Therefore $\mu(\rest{\fcal_C}{D})=\mu(\ai)$, $\rk{\rest{\fcal_C}{D}}=\rk{\ai}$ and $\eps(\rest{\fcal_C}{D})=\eps(\ai)$.
 Let
\begin{itemize}
 \item $\overline{\qcal}_D\doteqdot\rest{\qcal_C}{D}/T(\rest{\qcal_C}{D})$ and $\overline{\fcal}_D\doteqdot\ker(\rest{\ecal_\ai}{D}\to\overline{\qcal}_D)$, i.e., they fit in the exact sequence
\begin{equation*}
 0\longrightarrow\overline{\fcal}_D\longrightarrow\rest{\ecal_\ai}{D}\longrightarrow\frac{\rest{\qcal_C}{D}}{T(\rest{\qcal_C}{D})}\longrightarrow0;
\end{equation*}
 \item $\qcal_i\doteqdot\rest{\overline{\qcal}_D}{D_i}/T(\rest{\overline{\qcal}_D}{D_i})$ and $\fcal_i\doteqdot\ker(\rest{(\rest{\ecal_\ai}{D})}{D_i}\to\qcal_i)$, i.e., they fit in the exact sequence
\begin{equation*}
 0\longrightarrow\fcal_i\longrightarrow\rest{\ecal_\ai}{D_i}\longrightarrow\frac{\rest{\overline{\qcal}_D}{D_i}}{T(\rest{\overline{\qcal}_D}{D_i})}\longrightarrow0;
\end{equation*}
\end{itemize}
 
Then one gets
 \begin{itemize}
  \item $\rk{\ai}=\rk{\rest{\fcal_C}{D}}=\rk{\overline{\fcal}_D}=\rk{\rest{\overline{\fcal}_D}{D_i}}=\rk{\fcal_i}$ and $\rkq{\ai}=\rk{\rest{\qcal_C}{D}}=\rk{\overline{\qcal}_D}=\rk{\rest{\overline{\qcal}_D}{D_i}}=\rk{\qcal_i}$;
  \item $\muq(\ai)=\mu(\rest{\qcal_C}{D})\geq\mu(\overline{\qcal}_D)$ and $\mu(\ai)=\mu(\rest{\fcal_C}{D})\leq\mu(\overline{\fcal}_D)$;
  \item $\mu(\rest{\overline{\qcal}_D}{D_i})\geq\mu(\qcal_i)$ and $\mu(\rest{\overline{\fcal}_D}{D_i})\leq\mu(\fcal_i)$.
 \end{itemize}

Since $\rest{\ecal}{D}$ and $\overline{\qcal}_D$ are pure, and the sequences
  \begin{align*}
  &  0\longrightarrow\overline{\qcal}_D\longrightarrow\bigoplus_{i}\rest{(\overline{\qcal}_D)}{D_i}\longrightarrow\bigoplus_{i < j}\rest{(\overline{\qcal}_D)}{D_i\cap D_j}\longrightarrow0\\
  & 0\longrightarrow\rest{\ecal}{D}\longrightarrow\bigoplus_{i}\rest{(\rest{\ecal}{D})}{D_i}\longrightarrow\bigoplus_{i < j}\rest{(\rest{\ecal}{D})}{D_i\cap D_j}\longrightarrow0
  \end{align*}
  are exact modulo sheaves of dimension $n-3$, following the same calculations of Lemma $7.2.5$ in \cite{HLbook}, one gets that
\begin{align*}
 & \mu(\overline{\qcal}_D)=\sum_i\left( \mu(\rest{(\overline{\qcal}_D)}{D_i})-\frac{1}{2}\sum_{j\neq i}\left( \frac{\rk{\rest{(\overline{\qcal}_D)}{D_i\cap D_j}}}{\rkq{\ai}} - 1 \right)\ai_i\ai_j \right)\\
& \mu(\rest{\ecal}{D})=\sum_i\left( \mu(\rest{(\rest{\ecal}{D})}{D_i})-\frac{1}{2}\sum_{j\neq i}\left( \frac{\rk{\rest{(\rest{\ecal}{D})}{D_i\cap D_j}}}{\rk{\rest{\ecal}{D}}} - 1 \right)\ai_i\ai_j \right).
\end{align*}
and
\begin{equation*}
 \mu(\qcal_i)\leq\mu(\rest{(\overline{\qcal}_D)}{D_i})-\frac{1}{2}\sum_{j\neq i}\left( \frac{\rk{\rest{(\overline{\qcal}_D)}{D_i\cap D_j}}}{\rk{\overline{\qcal}_D}} - 1 \right)\ai_i\ai_j\\
\end{equation*}
Therefore $\mu(\overline{\qcal}_D)\geq\sum_i \mu(\qcal_i)$, $\deg(\rest{\ecal}{D})\leq\sum_i \deg(\rest{(\rest{\ecal}{D})}{D_i})$ and so easy calculations show that $\mu(\overline{\fcal}_D)\leq\sum_i \mu(\fcal_i)$.\\
 
Note that, since $T_F\doteqdot\overline{\fcal}_D/\rest{(\fcal_C)}{D}$ is pure torsion, $\rest{\varphi}{T_F}=0$ (see Remark \ref{rem-semistabilita}) and so $\eps(\rest{\varphi}{\rest{(\fcal_C)}{D}})=\eps(\rest{\varphi}{\overline{\fcal}_D})$. For the same reason $\eps(\rest{\varphi}{\rest{(\overline{\fcal}_D)}{D_i}})=\eps(\rest{\varphi}{\fcal_i})$. Moreover, if $\eps(\ai)=\eps(\rest{\fcal_C}{D})=0$ then obviously also $\epss{\fcal_i}=0$ for all $i$; conversely if $\eps(\ai)=1$ then there exists at least one $i$ such that $\epss{\fcal_i}=1$. Therefore $\sum_i\epss{\fcal_i}\geq\eps(\ai)\geq\epss{\fcal_i}$.\\ 

Therefore, defining $\epss{\qcal_i}=(1-\epss{\fcal_i})$ and $\epss{\overline{\qcal}_D}=(1-\epss{\overline{\fcal}_D})$ as in Remark \ref{rem-def-ddeg-al-quoziente}, thanks to the previous inequalities and considerations, one gets $\sum_i\epss{\qcal_i}\geq\epss{\overline{\qcal}_D}\geq\epss{\qcal_i}$ and so
\begin{equation*}
 \dmuq(\ai)\geq\dmu(\overline{\qcal}_D)\geq\sum_i\dmu(\qcal_i)\geq\sum_i\dmuq(\ai_i).
\end{equation*}

If $\dmuq(\ai)=\sum_i\dmuq(\ai_i)$ it follows that $\dmuq(\qcal_i)=\dmuq(\ai_i)$. Since $\dmuq(\ai)$ is the decorated slope of the minimal destabilizing quotient (i.e., its kernel is the maximal decorated destabilizing subsheaf), we have $\rkq{\ai}=\rk{\qcal_i}\geq\rkq{\ai_i}$ for all $i$.
\end{proof}

\begin{corollary}\label{cor-constant-slope}
$\rkq{\ai}$, $\frac{\eps^q(\ai)}{\ai}$, $\frac{\muq(\ai)}{\ai}$, $\frac{\dmuq(\ai)}{\ai}$, $\frac{\mu(\ai)}{\ai}$, $\dmu(\ai)$, $\rk{\ai}$ and $\eps(\ai)$ are constant for $\ai\gg0$.
\end{corollary}
\begin{proof}
 The quantities $\rkq{\ai}$ and $\frac{\muq(\ai)}{\ai}$ are constant as proved in \cite{HLbook} Corollary $7.2.6$. The same arguments show that $\frac{\dmuq(\ai)}{\ai}$ is constant as well. Therefore $\frac{\eps^q(\ai)}{\ai}$ has to be constant too and easy calculations show that also $\eps(\ai)$, $\frac{\mu(\ai)}{\ai}$ and $\dmu(\ai)$ are constant.
\end{proof}

\begin{corollary}
 For $\ai>>0$ or $\eps^q(\ai)=0$ and $\eps(\ai)=1$ or $\eps(\ecal_\ai,\varphi_\ai)=0$.
\end{corollary}
\begin{proof}
 Since $\frac{\eps^q(\ai)}{\ai}$ is definitively constant, $\eps^q(\ai)=0$ for $\ai>>0$. Since $\eps^q(\ai)=\eps(\ecal_\ai,\varphi_\ai)-\eps(\ai)$ or they are (definitively) both zero or both one.
\end{proof}

\begin{lemma}[Lemma 7.2.7 \cite{HLbook}]\label{lem-constant-lin-bun}
 There exist $\ai_0\in\nn$ and a line bundle $L\in\mbox{Pic}(X)$ such that $L_\ai\simeq L$ for any $\ai>\ai_0$.
\end{lemma}

In this way we have proved that for $\ai>>0$ an extension of $\det(\qcal_\ai)$ is of the form $L\boxtimes M_\ai$ with $L\in\pic(X)$ and $\deg(\rest{\qcal_\ai}{D})=\ai\deg(L)$ for any $D\in\vai$. Now we can state and prove the main theorem of this section:

\begin{theorem}\label{teo-rest}
Let $X$ be a smooth projective surface and $\amplebu{1}$ be a very ample line bundle. Let $(\ecal,\varphi)$ be a slope \fr-semistable decorated sheaf. Then there is an integer $\ai_0$ such that for all $\ai\geq \ai_0$ there is a dense open subset $U_\ai\subset |\amplebu{\ai}|$ such that for all $D\in U_\ai$ the divisor $D$ is smooth and $\rest{(\ecal,\varphi)}{D}$ is slope \fr-semistable.
\end{theorem}
\begin{proof}
We proof the theorem by reduction to absurd. Suppose the theorem is false: thanks to the previous constructions there exists a line bundle $L_{\ai}$ such that
\begin{equation*}
\frac{\deg(L_\ai)-a\deltabar\eps^q(\ai)}{\rkq{\ai}}<\dmu(\ecal)
\end{equation*}
and $1\leq\rkq{\ai}\leq\rk{\ecal}$. We recall that $\rkq{\ai}$ and $L_\ai$ are constant for $\ai$ greater than a certain constant $\ai_0$, so from now on we suppose that $\ai$ is so and we call $L_\ai=L$ and $\rkq{\ai}=\rqq$. We want to construct a rank $\rqq$ quotient $\qcal$ of $\ecal$ such that 
$\det(\qcal)=L$.\\
Let $\ai$ be a sufficiently large integer, $D\in U_\ai$ and let $(\sudi{\fcal},\rest{\varphi}{\sudi{\fcal}})$ be the maximal decorated destabilizing subsheaf of $\rest{(\ecal,\varphi)}{D}$ and $\sudi{\qcal}\doteqdot\coker(\sudi{\fcal}\hookrightarrow\rest{\ecal}{D})$ the associated minimal decorated destabilizing quotient. Put $L_D\doteqdot\det Q_D$ and note that $L_D=\rest{L}{D}$ (by uniqueness of the maximal destabilizing subsheaf and so of the minimal destabilizing quotient). The surjective morphism $\rest{\ecal}{D}\to Q_D$ induces a surjective homomorphism $\sigma_D\colon\Lambda^\rqq\rest{\ecal}{D}\to L_D$ and morphisms
\begin{equation*}
\rest{i}{D}\colon D\longrightarrow\mbox{Grass}(\rest{\ecal}{D},\rqq)\longrightarrow\pp(\Lambda^\rqq\rest{\ecal}{D}).
\end{equation*}
Consider the exact sequence
\begin{equation*}
\Hom(\Lambda^\rqq\ecal,L(-\ai))\to\Hom(\Lambda^\rqq\ecal,L)\stackrel{f}{\longrightarrow}\Hom(\Lambda^ \rqq\rest{\ecal}{D},\rest{L}{D})\to\mbox{Ext}^1(\Lambda^\rqq\ecal,L(-\ai)).
\end{equation*}
By Serre's theorem and Serre duality one has that for $i=0,1$ and $\ai\gg0$
\begin{equation*}
\mbox{Ext}^i(\Lambda^\rqq\ecal,L(-\ai))=H^{n-i}(X,\Lambda^\rqq\ecal\otimes L^{\vee}\otimes\omega_X(\ai))=0.
\end{equation*}
Hence if $\ai$ is big enough $f$ is bijective and $\sigma_D$ extends uniquely to a homomorphism $\sigma\in\Hom(\Lambda^\rqq\ecal,L)$. Using the same arguments of the final part of the proof of Theorem 7.2.1 in \cite{HLbook}, $\sigma$ induces a morphism $i\colon X\to \pp(\Lambda^\rqq\ecal)$ that factorize thorough $\mbox{Grass}(\ecal,\rqq)$ and so we obtain a quotient $q\colon\ecal\to\qcal$. Since $\rest{\det\qcal}{D}\equiv L_D=\rest{L}{D}$ for all $D\in U_\ai$, by Lemma 7.2.2 \cite{HLbook}, we get $L=\det\qcal$. Define $\fcal\doteqdot\ker(\ecal\to\qcal)$ and note that $\rest{\fcal}{D}=\sudi{\fcal}$. Finally, thanks to Fujita's vanishing theorem (\cite{Lazar} pg 66),
\begin{equation*}
H^i(X,\fcal\ab\otimes\linebu^{\vee}\otimes\omega_X(\ai))=0
\end{equation*}
for $i>0$ and $\ai$ big enough. Therefore
\begin{equation*}
 \mbox{Ext}^j(\fcal\ab,\linebu(-\ai))= H^{n-j}(X,\fcal\ab\otimes\linebu^{\vee}\otimes\omega_X(\ai))=0
\end{equation*}
for $j=0,1$. The same holds also for $\ecal$ and so the following diagram is commutative:
\begin{equation*}
\xymatrix{
 \Hom(\fcal\ab,\linebu) \ar@{<->}[r] & \Hom(\rest{\fcal\ab}{D},\rest{\linebu}{D})\\
 \Hom(\ecal\ab,\linebu) \ar[u]\ar@{<->}[r] & \Hom(\rest{\ecal\ab}{D},\rest{\linebu}{D}) \ar[u]\\
}
\end{equation*}
which proves that we can extend to all $\fcal\ab$ the morphism we have over $\rest{\fcal\ab}{D}$ in such a way that $\eps(\rest{\varphi}{\fcal})=\eps(\ai)$. By construction $(\fcal,\rest{\varphi}{\fcal})$ destabilizes, with respect to the slope \fr-semistability, the decorated sheaf $(\ecal,\varphi)$ and this contradicts the hypothesis.
\end{proof}
\section[Mehta-Ramanathan for slope $\kkvoid$-semistability]{Mehta-Ramanathan theorem for slope $\kkvoid$-semistable decorated sheaves of rank $2$ and $3$}\label{sec-mehta-ramanathan-2}
\begin{notation}
 Let $X$ be a smooth projective variety, $\amplebu{1}$ a fixed ample line bundle over $X$, $k$ an algebraic closed field of characteristic $0$, $S$ an integral $k$-scheme of finite type and $f:X\to S$ a projective flat morphism. Note that $\amplebu{1}$ is also $f$-ample.
\end{notation}

\begin{proposition}[\textbf{Properties of $\kk{\fcal}{\ecal}$}]\label{prop-proprieta-di-k}
Let $(\ecal,\varphi)$ be a decorated sheaf of type $(a,b,\linebu)$ and rank $r$. Let $\gcal,\fcal$ be subsheaves of $\ecal$. Then the following statements hold:
\begin{enumerate}
\item There exist an open subset $U\subseteq X$ and complex vector spaces $V'$ and $V$ of dimension $\rk{\fcal}$ and $r$ (respectively) such that $\rest{\fcal}{U}\simeq V'\otimes\ocal_U$, $\rest{\ecal}{U}\simeq V\otimes\ocal_U$ and $\kk{\fcal}{\ecal}=\kk{\rest{\fcal}{U}}{\rest{\ecal}{U}}$.
\item If there exists an open subset $U$ of $X$ such that $\rest{\fcal}{U}$ is isomorphic to $\rest{\gcal}{U}$, then $\kk{\fcal}{\ecal}=\kk{\gcal}{\ecal}$.
\item $\kk{\fcal+\gcal}{\ecal}\geq\max\{\kk{\fcal}{\ecal} \,,\, \kk{\gcal}{\ecal} \}$.
\item $\kk{\fcal\cap\gcal}{\ecal}\leq\min\{\kk{\fcal}{\ecal} \,,\, \kk{\gcal}{\ecal} \}$.
\item $\kk{\fcal}{\ecal}+\kk{\gcal}{\ecal}\geq\kk{\fcal+\gcal}{\ecal}$.
\item If $\kk{\fcal\cap\gcal}{\ecal}=\kk{\fcal\cap\gcal}{\fcal+\gcal}$ then
\begin{equation}\label{eq-disuguaglianza-tra-i-kk}
 \kk{\fcal+\gcal}{\ecal}+\kk{\fcal\cap\gcal}{\ecal}\leq\kk{\fcal}{\ecal}+\kk{\gcal}{\ecal}
\end{equation}
in particular 
\begin{equation*}
 \kk{\fcal+\gcal}{\fcal+\gcal}+\kk{\fcal\cap\gcal}{\fcal+\gcal}\leq\kk{\fcal}{\fcal+\gcal}+\kk{\gcal}{\fcal+\gcal}.
\end{equation*}
\end{enumerate}
\end{proposition}
\begin{proof}
 \begin{enumerate}
  \item Let $U_{\fcal}$ be a maximal open subset where $\fcal$ is locally free and admits a trivialization. Suppose that $\kk{\fcal}{\ecal}=k$, then there exist an open subset $U'\subseteq X$, $k$ local sections $f_1,\dots,f_k\in H^0(U',\rest{\fcal}{U'})$ and $a-k$ local sections $e_1,\dots,e_{a-k}\in H^0(U',\rest{\ecal}{U'})$ such that 
\begin{equation*}
 \varphi((f_1,\dots,f_t,e_1,\dots,e_{a-t})^{\oplus b})\neq0.
\end{equation*}
 Let $U\doteqdot U'\cap U_\fcal$, then $\rest{\ecal}{U}\simeq V\otimes\ocal_U$ and $\kk{\fcal}{\ecal}=\kk{\rest{\fcal}{U}}{\rest{\ecal}{U}}$.
  \item The statement follows directly from $(1)$.
  \item Follows from the fact that $\fcal,\gcal\subseteq\fcal+\gcal$.
  \item Since $\fcal\cap\gcal\subseteq\fcal,\gcal$, it is easy to see that the statement holds.
  \item Let $t=\kk{\fcal+\gcal}{\ecal}$ and $k=\kk{\fcal}{\ecal}$. Thanks to the first point, we can suppose that $\fcal\cap\gcal,\fcal,\gcal$ and $\fcal+\gcal$ are trivial sheaves. Let $f_1+g_1,\dots f_t+g_t$ and $e_1,\dots e_{a-t}$ be sections of $\fcal+\gcal$ and $\ecal$ respectively such that $f_i$ are sections of $\fcal$, $g_i$ are sections of $\gcal$ and $\varphi((f_1+g_1,\dots f_t+g_t,e_1,\dots e_{a-t})^{\oplus b})\neq 0$. Let $f=\sum f_i$, $g=\sum g_i$ and $e=\sum e_i$, then also $\varphi(((f+g)^{\otimes t}\otimes e^{\otimes a-t})^{\oplus b})\neq 0$. But
  \begin{equation*}
   (f+g)^{\otimes t}\otimes e^{\otimes a-t}=\left(\sum_{i=0}^{t} \binom{t}{i} f^{\otimes t-i}\otimes g^{\otimes i}\right)\otimes e^{\otimes a-t}.
  \end{equation*}
  Since $k\leq t$ there exists $i_0\geq0$ such that $t-i_0=k$. Then for any $0\leq i<i_0$ one has $\varphi((f^{\otimes t-i}\otimes g^{\otimes i}\otimes e^{\otimes a-t})^{\oplus b})=0$, since $\kk{\fcal}{\ecal}=k$ and $t-i>k$. Therefore 
  \begin{equation*}
   \varphi\left(\left(\sum_{i=i_0}^{t} \binom{t}{i} f^{\otimes t-i}\otimes g^{\otimes i}\otimes e^{\otimes a-t}\right)^{\oplus b}\right)\neq0
  \end{equation*}
  and so $\kk{\gcal}{\ecal}\geq i_0=t-k=\kk{\fcal+\gcal}{\ecal}-\kk{\fcal}{\ecal}$.
  \item Let $s=\kk{\fcal\cap\gcal}{\ecal}=\kk{\fcal\cap\gcal}{\fcal+\gcal}$. If $s=0$ there is nothing to prove. Otherwise, similarly to the proof of the previous point, we can choose a section $h$ of $\fcal\cap\gcal$ and sections $f$ and $g$ of $\fcal$ and $\gcal$ respectively such that $f+g$ is a section of $\fcal+\gcal$ and $\varphi((h^{\oplus s}\oplus(f+g)^{\oplus a-s})^{\oplus b})\neq 0$. In particular note that $\kk{\fcal+\gcal}{\ecal}=\kk{\fcal+\gcal}{\fcal+\gcal}=a$. Then is easy to see that $\kk{\gcal}{\ecal}\geq a - \kk{\fcal}{\ecal} + s$.
\end{enumerate}
\end{proof}

Let $(\ecal,\varphi)$ be a decorated sheaf and $\fcal$ a subsheaf of $\ecal$. As usual denote
\begin{align*}
 & \hilpol{\fcal}^{\kkvoid} \doteqdot \hilpol{\fcal}-\delta\kk{\fcal}{\ecal},\\
 & \hilpolred{\fcal}^\kkvoid \doteqdot \hilpol{\fcal}^\kkvoid/\rk{\fcal},\\
 & \deg^\kkvoid(\fcal)\doteqdot\deg(\fcal)-\deltabar\kk{\fcal}{\ecal},\\
 & \mu^\kkvoid(\fcal)\doteqdot\deg(\fcal)^\kkvoid/\rk{\fcal}.
\end{align*}
We recall that $(\ecal,\varphi)$ is \textbf{$\kkvoid$-(semi)stable}, respectively \textbf{slope $\kkvoid$-(semi)stable}, if and only if for any $\fcal\subset\ecal$
\begin{equation*}
 \hilpolred{\fcal}^\kkvoid\ppreceq\hilpolred{\ecal}^\kkvoid,
\end{equation*}
or
\begin{equation*}
 \mu^\kkvoid(\fcal)\lleq\mu^\kkvoid(\ecal),
\end{equation*}
respectively.
\subsection{Maximal destabilizing subsheaf}\label{maximal-dest-subsheaf2}
\begin{notation}
 In this section, unless otherwise stated, any decorated sheaf will have rank $r\leq3$.
\end{notation}

\begin{proposition}\label{prop-k-maximal-dest-subsheaf}
 Let $(\ecal,\varphi)$ be a decorated sheaf of type $(a,b,\linebu)$ and rank $r=2$ or $r=3$. If $(\ecal,\varphi)$ is not slope $\kkvoid$-semistable then there exists a unique, $\kkvoid$-slope-semistable subsheaf $\fcal$ of $\ecal$ such that:
\begin{enumerate}
 \item $\mu^\kkvoid(\fcal)\geq\mu^\kkvoid(\wcal)$ for any $\wcal\subset\ecal$.
 \item If $\mu^\kkvoid(\fcal)=\mu^\kkvoid(\wcal)$ then $\wcal\subseteq\fcal$.
\end{enumerate}
The subsheaf $\fcal$, with the induced morphism $\rest{\varphi}{\fcal}$, is called the \textbf{maximal slope $\kkvoid$-destabilizing subsheaf}.
\end{proposition}
\begin{proof}
 Define the following partial ordering on the set of decorated subsheaves of $\ecal$. Let $\fcal_1,\fcal_2$ be two subsheaves of $\ecal$; then
\begin{equation*}
 \fcal_1 \preccurlyeq^{\kkvoid} \fcal_2 \;\Longleftrightarrow\; \fcal_1\subseteq\fcal_2 \text{ and } \mu^\kkvoid(\fcal)\leq\mu^\kkvoid(\ecal).
\end{equation*}
The set of subsheaves of $\ecal$ with this ordering relation satisfies the hypotheses of Zorn's Lemma, so there exists a maximal element (not unique in general). Choose an element $\fcal$ in the following set:
\begin{equation*}
 \min_{\rk\gcal}\{\gcal\subset\ecal \;|\; \gcal \text{ is } \preccurlyeq^{\kkvoid}\text{-maximal}\}.
\end{equation*}
Then we claim that $(\fcal,\rest{\varphi}{\fcal})$ has the asserted properties.\\

By contradiction, suppose that there exists $\gcal\subset\ecal$ such that $ \mu^\kkvoid(\gcal)\geq\mu^\kkvoid(\fcal)$, i.e.,
\begin{equation*}
 \mu(\gcal)-\frac{\deltabar \, \kk{\gcal}{\ecal}}{\rkk{\gcal}}\geq\mu(\fcal)-\frac{\deltabar \, \kk{\fcal}{\ecal}}{\rkk{\fcal}}.
\end{equation*}

\textbf{Claim. } We can assume $\gcal\subseteq\fcal$ by replacing $\gcal$ by $\gcal\cap\fcal$.\\

Indeed, if $\gcal\not\subseteq\fcal$, $\fcal$ is a proper subsheaf of $\fcal+\gcal$ since (by the assumpions we made on the $\kkvoid$-slope of $\gcal$ and by maximality of $\fcal$) $\fcal\not\subset\gcal$. By maximality
\begin{equation*}
 \mu^\kkvoid(\fcal) > \mu^\kkvoid(\fcal+\gcal).
\end{equation*}
Using the exact sequence
\begin{equation*}
 0\longrightarrow\fcal\cap\gcal\longrightarrow\fcal\oplus\gcal\longrightarrow\fcal+\gcal\longrightarrow0
\end{equation*}
one finds, following calculations we made in the proof of Proposition \ref{prop-maximal-dest-subsheaf}, that
\begin{equation}\label{eq-deve-essere-minore-zero}
 \rkk{\fcal\cap\gcal}(\mu^\kkvoid(\gcal)-\mu^\kkvoid(\fcal\cap\gcal)) < \deltabar(\kk{\fcal+\gcal}{\ecal}+\kk{\fcal\cap\gcal}{\ecal}-\kk{\fcal}{\ecal}-\kk{\gcal}{\ecal}).
\end{equation}
Therefore \underline{if} $(\kk{\fcal+\gcal}{\ecal}+\kk{\fcal\cap\gcal}{\ecal}-\kk{\fcal}{\ecal}-\kk{\gcal}{\ecal})\leq0$ then $\mu^\kkvoid(\fcal\cap\gcal)\geq\mu^\kkvoid(\gcal)$ and the claim holds true.\\

First suppose that $r=2$.\\
Consider $\fcal\cap\gcal$. If $\rk{\fcal\cap\gcal}=0$ then $\kk{\fcal\cap\gcal}{\ecal}=0$ and, thanks to point $(5)$ of Proposition \ref{prop-proprieta-di-k}, the right part of equation \eqref{eq-deve-essere-minore-zero} is less or equal to zero, and the claim holds true. If $\rk{\fcal\cap\gcal}=2$ then $\rkk{\ecal}=\rkk{\gcal}=\rkk{\fcal}=\rkk{\fcal\cap\gcal}=\rkk{\fcal+\gcal}$, $\fcal$ coincides, up to a rank zero sheaf $\tcal=\ecal/\fcal$, with $\ecal$ and $\kk{\fcal+\gcal}{\ecal}=\kk{\fcal}{\ecal}=\kk{\gcal}{\ecal}=\kk{\fcal\cap\gcal}{\ecal}=\kk{\ecal}{\ecal}=a$. Since $\deg(\ecal)=\deg(\fcal)+\deg(\tcal)$ and $\kk{\fcal}{\ecal}=a=\kk{\ecal}{\ecal}$ one get that $\mu^\kkvoid(\fcal)\leq\mu^\kkvoid(\ecal)$ and $\fcal$ is not $\preccurlyeq^{\kkvoid}$-maximal, which is absurd. Therefore $\rk{\fcal\cap\gcal}=1$. If $\rkk{\fcal}$ of $\rkk{\gcal}$ are equal to $2$ then, as before, one easily gets that $\fcal$ is not maximal, against the assumptions. The only chance is that $\rkk{\fcal}=\rkk{\gcal}=\rkk{\fcal\cap\gcal}=\rkk{\fcal+\gcal}=1$ and so all these sheaves coincide with each other up to rank zero sheaves. Thus $\kk{\gcal}{\ecal}=\kk{\fcal}{\ecal}=\kk{\fcal\cap\gcal}{\ecal}=\kk{\fcal+\gcal}{\ecal}$. Therefore the inequality \ref{eq-disuguaglianza-tra-i-kk} holds true and $\mu^\kkvoid(\fcal\cap\gcal)\geq\mu^\kkvoid(\gcal)\geq\mu^\kkvoid(\fcal)$.\\

Now suppose that $r=3$.\\ 
If $\rk{\fcal\cap\gcal}=0$ then $\kk{\fcal\cap\gcal}{\ecal}=0$ and the right part of equation \eqref{eq-deve-essere-minore-zero} is less or equal to zero and the claim holds true. If $\rk{\fcal\cap\gcal}=3$ as before we easily fall in contradiction. If $\rk{\fcal}=3$ then $\kk{\fcal}{\ecal}=a$ and $\fcal$ coincides, up to a rank zero sheaf, with $\ecal$; so $\mu^\kkvoid(\fcal)\leq\mu^\kkvoid(\ecal)$ and $\fcal$ is not maximal, that is absurd. Similarly if $\rk{\gcal}=3$, then $\mu^\kkvoid(\fcal)\leq\mu^\kkvoid(\gcal)\leq\mu^\kkvoid(\ecal)$, that is again absurd. Therefore the possible cases are the following:\\

\noindent\\
  \begin{tabular}{|c|c|c|c|c|}
 \hline
 $\rk{\fcal\cap\gcal}$ & $\rk{\fcal}$ & $\rk{\gcal}$ & $\rk{\fcal+\gcal}$ & implies\\
 \hline
 $1$ & $1$ & $1$ & $1$ & $\kk{\fcal\cap\gcal}{\ecal}=\kk{\fcal}{\ecal}=\kk{\gcal}{\ecal}=\kk{\fcal+\gcal}{\ecal}$\\
 \hline
 $1$ & $1$ & $2$ & $2$ & $\kk{\fcal\cap\gcal}{\ecal}=\kk{\fcal}{\ecal}$ and $\kk{\gcal}{\ecal}=\kk{\fcal+\gcal}{\ecal}$\\
 \hline
 $1$ & $2$ & $1$ & $2$ & $\kk{\fcal\cap\gcal}{\ecal}=\kk{\gcal}{\ecal}$ and $\kk{\fcal}{\ecal}=\kk{\fcal+\gcal}{\ecal}$\\
 \hline
 $1$ & $2$ & $2$ & $3$ & $\kk{\fcal+\gcal}{\ecal}=\kk{\ecal}{\ecal}=a$\\
 \hline
  $2$ & $2$ & $2$ & $2$ & $\kk{\fcal+\gcal}{\ecal}=\kk{\ecal}{\ecal}=a$\\
 \hline
  \end{tabular}
\noindent\\
\noindent\\

The non-trivial cases are the following: $\rk{\fcal\cap\gcal}=1$ and $\rk{\fcal}=\rk{\gcal}=2$ or $\rk{\fcal\cap\gcal}=\rk{\fcal}=\rk{\gcal}=2$. In the first case $\rk{\fcal+\gcal}=3$ and so $\kk{\fcal+\gcal}{\ecal}=a$ and $\kk{\fcal\cap\gcal}{\fcal+\gcal}=\kk{\fcal\cap\gcal}{\ecal}$, in the second case $\rk{\fcal+\gcal}=2$ and so $\kk{\fcal\cap\gcal}{\ecal}=\kk{\fcal+\gcal}{\ecal}=\kk{\fcal}{\ecal}=\kk{\gcal}{\ecal}$. Therefore in both cases equation \eqref{eq-disuguaglianza-tra-i-kk} holds true and equation \eqref{eq-deve-essere-minore-zero} holds with the less or equal than zero. Then $\mu^\kkvoid(\fcal\cap\gcal)\geq\mu^\kkvoid(\gcal)\geq\mu^\kkvoid(\fcal)$ and the claim holds true.\\ 

Since we have proved the claim, the proof may continue as the proof of Proposition \ref{prop-maximal-dest-subsheaf}.
\end{proof}

\begin{proposition}
 Let $(\ecal,\varphi)$ be a decorated sheaf of type $(a,b,\linebu)$ and rank $r=2$ or $r=3$. If $(\ecal,\varphi)$ is not $\kkvoid$-semistable then exists a unique, $\kkvoid$-semistable subsheaf $\fcal$ of $\ecal$ such that:
\begin{enumerate}
 \item $\hilpolred{\fcal}^\kkvoid\preceq\hilpolred{\fcal}^\kkvoid$ for any $\wcal\subset\ecal$.
 \item If $\hilpolred{\fcal}^\kkvoid=\hilpolred{\wcal}^\kkvoid$ then $\wcal\subseteq\fcal$.
\end{enumerate}
The subsheaf $\fcal$, with the induced morphism $\rest{\varphi}{\fcal}$, is called the $\kkvoid$-\textbf{maximal destabilizing subsheaf}.
\end{proposition}
\begin{proof}
 The proof is similar to the proof of Proposition \ref{prop-k-maximal-dest-subsheaf}.
\end{proof}

\begin{remark}
 As in the \fr-semistable case, if $(\ecal,\varphi)$ is $\kkvoid$-semistable (resp. slope $\kkvoid$-semistable) the maximal $\kkvoid$-destabilizing (resp. slope $\kkvoid$-destabilizing) subsheaf coincide with $\ecal$.
\end{remark}
\subsection{Restriction theorem}\label{sec-restriction-theorem2}
In the previous section we proved that, given a decorated sheaf $(\ecal,\varphi)$ of rank less or equal to $3$, there exists a unique maximal $\kkvoid$-destabilizing subsheaf $(\fcal,\rest{\varphi}{\fcal})$. Since, as we noticed in Section \ref{sec-families-quotients}, there is a one-to-one correspondence between decorated subsheaves of $(\ecal,\varphi)$ and quotients of $\ecal$, we will call \textbf{minimal $\kkvoid$-destabilizing quotient} the (unique) sheaf $\qcal\doteqdot\text{coker}(\fcal\hookrightarrow\ecal)$ such that $\fcal$ is the maximal $\kkvoid$-destabilizing subsheaf.\\

In analogy with Section \ref{sec-mehta-ramanathan}, we will say that a decorated sheaf $(\ecal,\varphi)$ over a Noetherian scheme $Y$ is \textbf{flat over the fibre of a morphism} $f:Y\to S$ of finite type between Noetherian schemes if and only if
\begin{itemize}
 \item[-] $\ecal$ and $\linebu$ are flat families of sheaves over the fibre of $f:Y\to S$;
 \item[-] $\kk{\ecal_s}{\ecal_s}$ is locally constant as a function of $s$, where $\ecal_s\doteqdot\rest{\ecal}{f^{-1}(s)}$.
\end{itemize}
Note that the above conditions imply that the $\kkvoid$-Hilbert polynomials $\hilpol{\ecal_s}^\kkvoid$ are locally constant for $s\in S$. The converse holds only if $S$ is irreducible: i.e., asking that $\hilpol{\ecal_s}^\kkvoid$ is locally constant as function of $s$ is equivalent to ask that $\hilpol{\ecal_s}$ and $\kk{\ecal_s}{\ecal_s}$ are locally constant as functions of $s$.\\

If $(\fcal,\rest{\varphi}{\fcal})$ slope de-semistabilizes (resp. de-semistabilizes) $(\ecal,\varphi)$ then $\mu^\kkvoid(\fcal)>\mu^\kkvoid(\ecal)$ (resp $\hilpolred{\fcal}^\kkvoid\succ\hilpolred{\ecal}^\kkvoid$). Let $\qcal\doteqdot\text{coker}(\fcal\hookrightarrow\ecal)$; then
\begin{equation*}
 \mu(\qcal)<\mu(\ecal)-\deltabar\left(\frac{\kk{\fcal}{\ecal}}{\rkk{\qcal}}-\frac{a\rkk{\fcal}}{\rkk{\ecal}\rkk{\qcal}}\right),
\end{equation*}
or
\begin{equation*}
 \hilpolred{\qcal}\prec\hilpolred{\ecal}-\delta\left(\frac{\kk{\fcal}{\ecal}}{\rkk{\qcal}}-\frac{a\rkk{\fcal}}{\rkk{\ecal}\rkk{\qcal}}\right),
\end{equation*}
respectively.

Define $C_{\kkvoid,i}\doteqdot\left(\frac{i}{\rkk{\qcal}}-\frac{a\rkk{\fcal}}{\rkk{\ecal}\rkk{\qcal}}\right)$ for $i=0,\dots,a$. Let $(\ecal,\varphi)$ be a flat family of decorated sheaves over the fibre of a projective morphism $f:X\to S$. Let $\hilpol{}=\hilpol{\ecal_s}$ and $\hilpolred{}=\hilpolred{\ecal_s}$ the Hilbert polynomial and, respectively, the reduced Hilbert polynomial of $\ecal_s$ (which are constant because the family is flat over $S$). Define:
\begin{enumerate}
 \item $\mathfrak{F}_{\kkvoid}$ as the family on $X$ parameterized by $S$ of saturated subsheaves $\fcal\hookrightarrow\ecal_s$ such that the induced torsion free quotient $\ecal_s\twoheadrightarrow\qcal$ satisfy $\mu(\qcal)\leq\mu(\ecal_s)+\deltabar C_{\kkvoid,0}$;
 \item $\mathfrak{F}_{\kkvoid,i}$ as the family of decorated subsheaves $(\fcal,\rest{\varphi}{\fcal})\hookrightarrow(\ecal_s,\rest{\varphi}{\ecal_s})$ such that:\begin{itemize}
	 \item[-] $\kk{\fcal}{\ecal_s}=i$;
	 \item[-] $\hilpolred{\fcal}^\kkvoid\succ\hilpolred{\ecal}^\kkvoid$ (i.e., $\hilpolred{\qcal}\prec\hilpolred{}+\delta C_{\kkvoid,i}$ with $\qcal\doteqdot\text{coker}(\fcal\hookrightarrow\ecal_s)$);
	 \item[-] $\fcal$ is a saturated subsheaf of $\ecal_s$,
      \end{itemize}
for $i=0,\dots,a$.
\end{enumerate}

It is easy to see that these families are bounded and therefore, using the same techniques used in Section \ref{sec-opennes-semistab-condition} one can prove that $\kkvoid$-semistability is open. More precisely
\begin{proposition}
  Let $f:X\to S$ be a projective morphism of Noetherian schemes and let $(\ecal,\varphi)$ be a flat family of decorated sheaves over the fibre of $f$. The set of points $s\in S$ such that $(\ecal_s,\rest{\varphi}{\ecal_s})$ is $\kkvoid$-(semi)stable with respect to $\delta$ is open in $S$.
\end{proposition}
\begin{proof}
 Let
\begin{equation}
 A\doteqdot\{P''\in\qq[x] \,|\, \exists\, s\in S, \exists\, q:\ecal_s\twoheadrightarrow\qcal \text{ such that }\hilpol{\qcal}=P'' \text{ and } \ker(q)\in\mathfrak{F}_\kkvoid\}
\end{equation}
and, for $i=0,\dots,a$,
\begin{equation*}
 A_{\kkvoid,i}\doteqdot\{P''\in\qq[x] \,|\, \exists\, s\in S, \exists\, q:\ecal_s\twoheadrightarrow\qcal \text{ such that }\hilpol{\qcal}=P'' \text{ and } \ker(q)\in\mathfrak{F}_{\kkvoid,i}\}.
\end{equation*}
 Then, using the same techniques used in the proof of Proposition \ref{prop-eps-semistability-open-condition}, one concludes the proof.
\end{proof}

Thanks to the previous results and using the same arguments as in the proof of Theorem \ref{teo-rel-max-dest}, it is easy to see that the following theorem holds true:
\begin{theorem}[\textbf{Relative maximal $\kkvoid$-destabilizing subsheaf}]
 Let $X$, $\amplebu{1}$, $S$ and $f:X\to S$ as before. Let $(\ecal,\varphi)$ be a decorated sheaf of rank $r\leq3$. Then there is an integral $k$-scheme $T$ of finite type, a projective birational morphism $g:T\to S$, a dense open subset $U\subset T$ and a flat quotient $\qcal$ of $\ecal_T$ such that for all points $t\in U$, $\fcal_t\doteqdot\ker(\ecal_t\twoheadrightarrow\qcal_t)$ with the induced morphism $\rest{\varphi_t}{\fcal_t}$ is the maximal $\kkvoid$-destabilizing subsheaf of $(\ecal_t,\varphi_t)$ or $\qcal_t=\ecal_t$.\\
 Moreover the pair $(g,\qcal)$ is universal in the sense that if $g':T'\to S$ is any dominant morphism of $k$-integral schemes and $\qcal'$ is a flat quotient of $\ecal_T'$, satisfying the same property of $\qcal$, there is an $S$-morphism $h:T'\to T$ such that $h^*_X(\qcal)=\qcal'$.
\end{theorem}

Finally, following the costructions made in Section \ref{sec-restriction-theorem} and replacing $\kkvoid$ with $\eps$, one can prove the following
\begin{theorem}[\textbf{Mehta-Ramanathan for slope $\kkvoid$-semistable decorated sheaves}]
Let $X$ be a smooth projective surface and $\amplebu{1}$ be, as usual, a very ample line bundle. Let $(\ecal,\varphi)$ be a slope $\kkvoid$-semistable decorated sheaf of rank $r\leq3$. Then there is an integer $\ai_0$ such that for all $\ai\geq \ai_0$ there is a dense open subset $U_\ai\subset |\amplebu{\ai}|$ such that for all $D\in U_\ai$ the divisor $D$ is smooth and $\rest{(\ecal,\varphi)}{D}$ is slope $\kkvoid$-semistable.
\end{theorem}
\subsection{Decorated sheaves of rank $2$}\label{sec-decorated-sheaves-rank-2}
\begin{lemma}\label{lemma-decorati-rk2-kss-uguale-ss}
 Let $(\ecal,\varphi)$ be a decorated sheaf of rank $r=2$. Then the following conditions are equivalent:
\begin{itemize}
 \item $(\ecal,\varphi)$ is (semi)stable (in the sense of Definition \ref{def-ss});
 \item $(\ecal,\varphi)$ is $\kkvoid$-(semi)stable.
\end{itemize}
\end{lemma}
\begin{proof}
 Since the rank of $\ecal$ is equal to $2$ all filtrations of $\ecal$ are non-critical and of lengh one. Then the statement follows from the fourth point of Remark \ref{rem-ss-uguale-kss-piu-condition-su-filtr-critiche}.
\end{proof}

Thanks to the previous Lemma all results in the previous section holds true for rank $2$ semistable decorated sheaves of type $(a,b,\linebu)$. In particular, for such objects
\begin{itemize}
 \item we have found the maximal destabilizing subsheaf and the relative maximal destabilizing subsheaf;
 \item we have provided the Harder-Narasimhan filtration and the relative Harder-Narasimhan filtration;
 \item we have proved that the semistability condition is open;
 \item we have proved a Mehta-Ramanathan's like theorem.
\end{itemize}
\section{Further remarks}\label{sec-remarks}
\begin{enumerate}
 \item In Section \ref{sec-mehta-ramanathan} we never used that $\linebu$ is of rank $1$ nor that it is a vector bundle. We only used that it is a pure dimensional torsion free sheaf (of positive rank). Therefore all results in this chapter can be easily generalized for pairs $(\ecal,\varphi)$ of type $(a,b,c,\linebu)$ where $\ecal$ and $\linebu$ are torsion free sheaves over $X$, $a,b,c$ are positive integers and
\begin{equation*}
 \varphi\colon\ecal\ab\longrightarrow\det(\ecal)^{\otimes c}\otimes\linebu.
\end{equation*}

 \item Let $(\acal,\varphi)$ be a decorated sheaf of rank $r>0$. Define $\wtild{\acal}\doteqdot\acal\ab$; then the pair $(\wtild{\acal},\varphi)$ can be regarded as a framed sheaf. Recall that a framed sheaf $(A,\alpha)$ of positive rank and with nonzero morphism $\alpha$ is slope semistable with respect to $\wtild{\delta}$ if and only if for any $F\subset A$
\begin{equation*}
 \mu(F)-\frac{\wtild{\delta} \; \eps(\rest{\alpha}{F})}{\rk{F}} \leq \mu(A)-\frac{\wtild{\delta}}{\rk{A}}
\end{equation*}
Suppose now that $(\wtild{\acal},\varphi)$ is frame semistable with respect to $\wtild{\delta}$, then $(\acal,\varphi)$ is slope \fr-semistable with respect to $\hat{\delta}\doteqdot \frac{\wtild{\delta}}{a^2\,b\,r^{a-1}}$. In fact if
\begin{equation*}
 \rk{\acal\ab} = b\,r^{a} \qquad \deg(\acal\ab) = a\,b\,r^{(a-1)}\deg(\acal)
\end{equation*}
and so if $\fcal$ is a subsheaf of $\acal$ then
\begin{equation*}
 \mu(\fcal\ab)-\frac{\wtild{\delta} \; \eps(\rest{\varphi}{\fcal\ab})}{\rkk{\fcal\ab}} \leq  \mu(\wtild{\acal})-\frac{\wtild{\delta}}{\rkk{\wtild{\acal}}}
\end{equation*}
which implies that
\begin{equation*}
a \mu(\fcal)-\frac{\wtild{\delta} \; \eps(\rest{\varphi}{\fcal\ab})}{b(\rkk{\fcal})^{a}} \leq a \mu(\acal)-\frac{\wtild{\delta}}{b r^a}\\
\end{equation*}
and so
\begin{equation*}
\dmu(\fcal)=\mu(\fcal)-a\underbrace{\frac{\wtild{\delta}}{a^2 b(\rkk{\fcal})^{a-1}}}_{=\hat{\delta}} \; \frac{\eps(\rest{\varphi}{\fcal\ab})}{\rkk{\fcal}} \leq  \mu(\acal)-a\underbrace{\frac{\wtild{\delta}}{a^2 b r^{a-1}}}_{=\hat{\delta}} \frac{1}{r}=\dmu(\ecal).
\end{equation*}
Since the subsheaves of $\acal$ correspond to subsheaves of $\acal\ab$ but this correspondence is \textit{not} surjective, the converse does not hold in general but only if $a=1$ ($b$ and $c$ generic). Thanks to the previous calculations and to Proposition \ref{prop-eps-implica-ss-implica-k}, one has that
\begin{align*}
  (\acal\ab,\varphi) \;\; \wtild{\delta}\text{ frame slope (semi)stable} & \Rightarrow (\acal,\varphi) \;\; \hat{\delta}\text{ slope \fr-(semi)stable}\\ & \Rightarrow \hat{\delta}\text{ slope (semi)stable}\\
  & \Rightarrow \hat{\delta} \;\; \kkvoid\text{-(semi)stable}.
\end{align*}
Replacing $\ddeg$ by $\hilpol{}^\eps$ and $\dmu$ by $\hilpolred{}^\eps$, similar calculations show that the same result hols also for semistability and not only for slope semistability.

\end{enumerate}
%
%
%
\bigskip
\textbf{Acknowledgments.} I would like to thank Professor Ugo Bruzzo for discussions and comments.\\

\end{document}